\documentclass[final,US,11pt,reqno,bibtex,nopagebackref,notcite,notref]{amsart}
\usepackage{liao,comment}

\DeclareMathOperator{\im}{im}

\DeclareMathOperator{\Hom}{Hom}

\DeclareMathOperator{\End}{End}

\DeclareMathOperator{\id}{id}
\DeclareMathOperator{\rk}{rk}
\DeclareMathOperator{\CE}{CE}

\DeclareMathOperator{\Bott}{Bott}
\DeclareMathOperator{\At}{At}
\DeclareMathOperator{\Toddcocycle}{td}
\DeclareMathOperator{\Todd}{Td}
\DeclareMathOperator{\ver}{vertical}
\DeclareMathOperator{\hor}{horizontal}
\DeclareMathOperator{\Ber}{Ber}
\DeclareMathOperator{\str}{str}
\DeclareMathOperator{\tr}{tr}

\newcommand{\ZZ}{\mathbb{Z}} 
 
\newcommand{\RR}{\mathbb{R}} 
\newcommand{\CC}{\mathbb{C}} 
\newcommand{\KK}{\Bbbk}
\newcommand{\TT}{\mathbb{T}} 
\newcommand{\into}{\hookrightarrow}
\newcommand{\onto}{\twoheadrightarrow}
\newcommand{\xto}[1]{\xrightarrow{#1}}
\newcommand{\dual}{^{\vee}}
\newcommand{\inv}{^{-1}}

\newcommand{\XX}{{\mathfrak{X}}}

\newcommand{\frakg}{\mathfrak g}

\newcommand{\cL}{\mathcal{L}}
\newcommand{\cE}{\mathcal{E}}
\newcommand{\cM}{\mathcal{M}}

\newcommand{\cQ}{\mathcal{Q}}

\newcommand{\cO}{\mathcal{O}}
\newcommand{\cT}{\mathcal{T}}
\newcommand{\cF}{\mathcal{F}}

\newcommand{\sA}{\mathscr{A}}
\newcommand{\sB}{\mathscr{B}}
\newcommand{\sC}{\mathscr{C}}
\newcommand{\sD}{\mathscr{D}}

\newcommand{\pair}[2]{\langle #1, #2 \rangle}

\newcommand{\argument}{-}
\newcommand{\px}[1]{\frac{\partial}{\partial x^{#1}}}
\newcommand{\peta}[1]{\frac{\partial}{\partial \eta^{#1}}}
\newcommand{\pshift}{\mathrm{s}}
\newcommand{\nshift}{\mathfrak{s}}
\newcommand{\pairAtiyah}[1]{R^{#1}_{1,1}}
\newcommand{\cLAtiyah}{\At_\cL^{\nabla^\cL}}
\newcommand{\cLAtiyahProj}{\bar\At_\cL^{\nabla}}

\usepackage{scalerel}

\newcommand{\pa}{p_{ \scaleto{A}{3.5pt}}}
\newcommand{\ia}{i_{ \scaleto{A}{3.5pt}}}
\newcommand{\pb}{p_{ \scaleto{B}{3.5pt}}}
\newcommand{\ib}{i_{ \scaleto{B}{3.5pt}}}
\newcommand{\hpt}{\widetilde{p_{ \scaleto{A}{3.5pt}}}}
\newcommand{\incl}{\widetilde{\imath_{ \scaleto{A}{3.5pt}}}}
\newcommand{\anchorL}{\rho_{ \scaleto{L}{3.5pt}}}

\newcommand{\anchor}{\rho}
\newcommand{\christoffelLL}{{\Gamma}}
\newcommand{\christoffelLA}{{^{\scaleto{L}{3.5pt}}_{\scaleto{A}{3.5pt}}\Gamma}}
\newcommand{\christoffelAL}{{^{\scaleto{A}{3.5pt}}_{\scaleto{L}{3.5pt}}\Gamma}}
\newcommand{\christoffelAA}{{^{\scaleto{A}{3.5pt}}_{\scaleto{A}{3.5pt}}\Gamma}}
\newcommand{\ssym}{\widetilde{\mathrm{sym}}}
\newcommand{\Syms}[2]{S^{#1}(#2[-1])}

\title{Atiyah classes and Todd classes of pullback dg Lie algebroids
associated with Lie pairs}

\thanks{Research partially supported by KIAS
Individual Grant MG072801 and MoST/NSTC Grant
110-2115-M-007-001-MY2 and 112-2115-M-007-016-MY3.}

\author{Hsuan-Yi Liao}
\address{Department of Mathematics, National Tsing Hua University}
\email{hyliao@math.nthu.edu.tw}

\begin{document}

\begin{abstract}
For a Lie algebroid $L$ and a Lie subalgebroid $A$, i.e.\ a Lie pair $(L,A)$,
we study the Atiyah class and the Todd class of the pullback dg
(i.e.\ differential graded) Lie algebroid $\pi^! L$ of $L$
along the bundle projection $\pi:A[1] \to M$ of the shifted vector bundle $A[1]$.
Applying the homological perturbation lemma, we provide a new construction
of Sti\'{e}non--Vitagliano--Xu's contraction relating the cochain complex
$\big(\Gamma(\pi^! L),\cQ\big)$ of sections of $\pi^! L$ to the Chevalley--Eilenberg
complex $(\Gamma(\Lambda^\bullet A^\vee\otimes(L/A)),d^{\Bott})$ of the Bott
representation. Using this contraction, we construct two isomorphisms:
the first identifies the cohomology of the cochain complex
$(\Gamma((\pi^! L)^\vee\otimes\End(\pi^! L)),\cQ)$ with the Chevalley--Eilenberg
cohomology $H^\bullet_{\CE}(A,(L/A)^\vee\otimes\End(L/A))$ arising from the
Bott representation, while the second identifies the cohomologies
$H^\bullet(\Gamma(\Lambda(\pi^! L)^\vee),\cQ)$ and $H^\bullet_{\CE}(A,\Lambda(L/A)^\vee)$.
We prove that this pair of isomorphisms identifies the Atiyah class and the Todd class
of the dg Lie algebroid $\pi^! L$ with the Atiyah class and the Todd class of the Lie pair $(L,A)$, respectively.
\end{abstract}

\maketitle

\tableofcontents

\section*{Introduction}

In \cite{MR86359}, Atiyah introduced a characteristic class, now known as the Atiyah class,
to characterize the obstruction to the existence of holomorphic connections
on a holomorphic vector bundle. Decades later, Kapranov \cite{MR1671737} showed
that the Atiyah class of a K\"{a}hler manifold $X$ induces an $L_\infty[1]$ algebra
structure on the Dolbeault complex $\Omega^{0,\bullet}(T_X^{1,0})$.
Kapranov's result was later shown to hold for all complex manifolds
\cite{MR4271478,MR2989383}. The Atiyah class plays an important role in the construction
of Rozansky--Witten invariants \cite{MR1671737,MR1671725}.
In addition to Rozansky--Witten theory, Kontsevich \cite{MR2062626} brought to light
a deep relation between the Todd class of complex manifolds and the Duflo element
of Lie algebras. See \cite{MR3754617,MR4584414} for a unified framework
for deriving the Duflo--Kontsevich isomorphism for Lie algebras and Kontsevich's
isomorphism for complex manifolds \cite{MR2062626}.

The works of Kapranov \cite{MR1671737} and Kontsevich \cite{MR2062626} have led
to many new developments in the theory of Atiyah classes. For instance,
see \cite{MR3439229,MR4271478,MR3319134,MR2646112,2011arXiv1112.0816C,MR3331615,MR4393962}. 
In the present paper, we are particularly interested in Chen--Sti\'{e}non--Xu's approach
via Lie pairs \cite{MR3439229} and Mehta--Sti\'{e}non--Xu's approach
via dg Lie algebroids \cite{MR3319134}.
By a Lie pair $(L,A)$, we mean a pair consisting of a Lie algebroid $L$
and a Lie subalgebroid $A$ of $L$ over a common base manifold $M$.
In \cite{MR3439229}, Chen, Sti\'{e}non and Xu introduced the Atiyah class
of a Lie pair $(L,A)$, which captures the obstruction to the existence
of compatible $L$-connections on $L/A$ extending the Bott $A$-connection.
Chen--Sti\'{e}non--Xu's theory includes the Atiyah class of complex manifolds
and the Molino class \cite{MR281224} of foliations as special cases.
In a different direction, Mehta, Sti\'{e}non and Xu \cite{MR3319134}
introduced the Atiyah class of a dg vector bundle $\cE$
relative to a dg Lie algebroid $\cL$, which measures the obstruction
to the existence of $\cL$-connections on $\cE$ which are compatible with the dg structure.

In fact, Mehta--Sti\'{e}non--Xu's approach \cite{MR3319134} is more general
than Chen--Sti\'{e}non--Xu's approach \cite{MR3439229}. 
In \cite{MR3724780}, Batakidis and Voglaire constructed a dg manifold structure
on $L[1]\oplus L/A$ --- which was independently constructed by Sti\'{e}non
and Xu \cite{MR4150934} --- by Fedosov's iteration method and they proved that,
for a matched pair of Lie algebroids, the Atiyah class of the Fedosov dg Lie algebroid
$\cF \to L[1]\oplus L/A$ can be identified with the Atiyah class of the Lie pair $(L,A)$.
In \cite[Section~1.7]{MR3964152}, Sti\'{e}non, Xu and the author obtained
an analogous identification for arbitrary Lie pairs. 
In \cite{MR3877426}, Chen, Xiang and Xu constructed different quasi-isomorphisms
for the Lie pairs $(T_M,F)$ arising from integrable distributions.
They proved that the Atiyah class of the dg manifold $F[1]$ associated
with a foliation corresponds to the Atiyah class of the Lie pair $(T_M,F)$
under a natural quasi-isomorphism.
In the present paper, we prove a theorem analogous to Chen--Xiang--Xu's theorem
in full generality. Namely, for an arbitrary Lie pair $(L,A)$ over a manifold $M$,
we investigate the pullback dg Lie algebroid $\cL = \pi^!L \to A[1]$
along the projection $\pi:A[1] \to M$ whose associated cohomology
$H^\bullet\big(\Gamma(\cL\dual \otimes \End \cL),\cQ\big)$ is isomorphic
to the Chevalley--Eilenberg cohomology
$H_{\CE}^\bullet\big(A,(L/A)\dual \otimes \End (L/A)\big)$,
and we prove that the Atiyah class of the dg Lie algebroid $\cL$ is identified
with the Atiyah class of the Lie pair $(L,A)$ under this isomorphism.
The dg Lie algebroid $\cL \to A[1]$ we consider here is much simpler
than the Fedosov dg Lie algebroid $\cF \to L[1] \oplus L/A$. 

Let $(L,A)$ be a Lie pair, and $B$ be the quotient vector bundle $B = L/A$. 
In \cite{2022arXiv221016769S}, Sti\'{e}non, Vitagliano and Xu studied
the pullback Lie algebroid $\cL = \pi^!L \to A[1]$ and proved that $\cL $
is a dg Lie algebroid over the dg manifold $(A[1],d_A)$, where $d_A$
is the Chevalley--Eilenberg differential of the Lie algebroid $A$.
See \cite[Section~4.2]{MR2157566} for the definition of pullback Lie algebroids.
Furthermore, by choosing a splitting of the short exact sequence of vector bundles
\begin{equation}\label{eq_intro:splitting}
\begin{tikzcd}
0 \ar[r] &  A \ar[r, hook, "\ia"'] & \ar[l, bend right, dashed,"\pa"'] L \ar[r, two heads,"\pb"'] & \ar[l, bend right, dashed,"\ib"'] B \ar[r] & 0
,\end{tikzcd}
\end{equation}
Sti\'{e}non, Vitagliano and Xu proved the following 
\begin{trivlist}
\item {\bf Theorem A.}
{\it By choosing a splitting \eqref{eq_intro:splitting}, one has the contraction data 
\begin{equation}\label{eq_inrto:piL_contraction}
\begin{tikzcd}
\big(\Gamma(\pi^* B),d^{\Bott}\big)
\arrow[r,  shift left,hook] &  \big(\Gamma(\cL), \cQ \big)
\arrow[l, shift left,two heads] \arrow[loop,out=12,in= -12,looseness = 3]
\end{tikzcd} 
\end{equation}
over the dg ring $(\Gamma(\Lambda^\bullet A\dual),d_A)$, where $d^{\Bott}$ is the Bott differential, and $\cQ$ is induced by the dg Lie algebroid structure of $\cL$.} 
\end{trivlist}
Sti\'{e}non--Vitagliano--Xu's method is computational and heavily based
on explicit formulas. Here, we give a more conceptual proof of Theorem~A
relying on the homological perturbation lemma (Theorem~\ref{thm:HPLCptForm}).
It is immediate that a splitting of a short exact sequence of vector spaces
(see \eqref{saturn})
induces a contraction data (see \eqref{eq:VecSpContraction}).
Applying this vector space construction fiberwisely
to $\cL = T_{A[1]}\times_{T_M} L \cong \pi^*A[1] \oplus \pi^*L$,
from a splitting \eqref{eq_intro:splitting}, we obtain a contraction data 
\begin{equation}\label{eq_intro:BasicContraction}
\begin{tikzcd}
\big(\Gamma(\pi^* B),0\big)
\arrow[r, " \ib", shift left,hook] &  \big(\Gamma(\cL), \incl\big)
\arrow[l, " \pb", shift left,two heads] \arrow[loop, "\hpt",out=12,in= -12,looseness = 3]
\end{tikzcd}
\end{equation}
over $\Gamma(\Lambda^\bullet A\dual)$.
Then we perturb \eqref{eq_intro:BasicContraction} by $\cQ -\incl$
and prove that the perturbed contraction coincides with Sti\'{e}non--Vitagliano--Xu's
contraction \eqref{eq_inrto:piL_contraction}.
See Proposition~\ref{prop:BasicContractionLiePair}
and Theorem~\ref{prop:ContPullbackLieAbd} for details.

In order to state our main theorem, we briefly review the Atiyah class
and the Todd class of a Lie pair and of a dg Lie algebroid.
Let $(L,A)$ be a Lie pair, and $\nabla$ be an $L$-connection on $B=L/A$
extending the Bott connection. The curvature of $\nabla$ induces
a Chevalley--Eilenberg cocycle
$\pairAtiyah\nabla \in \Gamma(A\dual \otimes B\dual \otimes \End B)$.
The Atiyah class of the Lie pair $(L,A)$ is the cohomology class
$\alpha_{L/A} = [\pairAtiyah\nabla] \in H^1_{\CE}(A, B\dual \otimes \End B)$,
which is independent of the choice of $L$-connection $\nabla$.
The Todd class of the Lie pair $(L,A)$ is the cohomology class
$$
\Todd_{L/A} = \det\left( \dfrac{\alpha_{L/A}}{1-e^{-\alpha_{L/A}}} \right)
\in \bigoplus_{k=0}^\infty H_{\CE}^k( A, \Lambda^k B\dual).
$$

Let $\cL \to \cM$ be a dg Lie algebroid equipped with the homological vector field $\cQ$,
and let $\tilde\nabla$ be an $\cL$-connection on $\cL$. The Lie derivative
$\At_\cL^{\tilde\nabla} = L_\cQ(\tilde\nabla)$ of the connection $\tilde\nabla$
is a $\cQ$-cocycle $\At_\cL^{\tilde\nabla} \in \Gamma(\cL\dual\otimes \End \cL)$.
The induced cohomology class $\alpha_\cL = [\At_\cL^{\tilde\nabla}]
\in H^1\big(\Gamma(\cL\dual \otimes \End \cL), \cQ\big)$ is independent
of the choice of $\cL$-connection $\tilde\nabla$ and is called the Atiyah class
of the dg Lie algebroid $\cL$.
The Todd class of the dg Lie algebroid $\cL$ is the cohomology class
$$
\Todd_{\cL} = \Ber\left( \dfrac{\alpha_{\cL}}{1-e^{-\alpha_{\cL}}} \right)
\in \prod_{k=0}^\infty H^k(\Gamma(\Lambda^k \cL\dual), \cQ),
$$
where $\Ber$ denotes the Berezinian.

According to general algebraic constructions (Section~\ref{sec:HomTensorConstruction}),
the contraction \eqref{eq_inrto:piL_contraction} induces the contraction data
\begin{equation}\label{eq_into:AtiyahContraction}
\begin{tikzcd}
\Big(\Gamma\big(\pi^*(B\dual \otimes \End B)\big),d^{\Bott}\Big)
\arrow[r, shift left,hook,"\cT^1_2"] &  \Big(\Gamma( \cL\dual \otimes \End \cL), \cQ \Big)
\arrow[l, " \Pi^1_2", shift left,two heads] \arrow[loop, "H^1_2",out=7,in= -7,looseness = 3]
\end{tikzcd}
\end{equation}
and
\begin{equation}\label{eq_intro:WedgeContraction}
\begin{tikzcd}
\big(\Gamma\big(\pi^* (\Lambda B\dual)\big),d^{\Bott} \big)
\arrow[r, " \cT_\Lambda", shift left,hook] &  \big(\Gamma(\Lambda \cL\dual ), \cQ \big)
\arrow[l, " \Pi_\Lambda", shift left,two heads] \arrow[loop, "H_\Lambda",out=8,in= -8,looseness = 3]
\end{tikzcd},
\end{equation}
where $\cL$ is the pullback dg Lie algebroid $\pi^!L \to A[1]$.
In particular, the projection maps $\Pi^1_2$ and $\Pi_\Lambda$ are quasi-isomorphisms. 
Our main theorem is the following
\begin{trivlist}
\item {\bf Theorem B.}
{\it Given any Lie pair $(L,A)$, the isomorphisms
\begin{gather*}
(\Pi^1_2)_*: H^1\big(\Gamma(\cL\dual \otimes \End \cL), \cQ\big) \xto\cong H^1_{\CE}(A, B\dual \otimes \End B), \\
(\Pi_\Lambda)_*:\prod_{k=0}^\infty H^k(\Gamma(\Lambda^k \cL\dual), \cQ) \xto\cong \bigoplus_{k=0}^\infty H_{\CE}^k( A, \Lambda^k B\dual)
\end{gather*}
send the Atiyah class and the Todd class of the dg Lie algebroid $\cL=\pi^! L$
to the Atiyah class and the Todd class of the Lie pair $(L,A)$, respectively:
\begin{gather*}
(\Pi^1_2)_*(\alpha_\cL) = \alpha_{L/A},\\
(\Pi_\Lambda)_*(\Todd_\cL) = \Todd_{L/A}.
\end{gather*}
}
\end{trivlist}

Let $L=T_M$ be the tangent bundle of a manifold $M$, and let $A=F \subset T_M$ be a Lie subalgebroid whose sections form an integrable distribution. The pullback dg Lie algebroid $\pi^! T_M $ can be identified with the dg Lie algebroid $T_{F[1]}$ of tangent bundle equipped with the Lie derivative $L_{d_F}$ with respect to the Chevalley--Eilenberg differential $d_F$ of $F$. In this case, the Atiyah class and the Todd class of the dg Lie algebroid $T_{F[1]}$ are exactly the Atiyah class and the Todd class of the dg manifold $F[1]$, respectively, and we recover Chen--Xiang--Xu's theorems in \cite{MR3877426} by Theorem~B.

\subsection*{Notations and conventions}
We fix a base field $\KK = \RR$ or $\CC$ in this paper.
The notation $C^\infty(M)=C^\infty(M,\KK)$ refers to the algebra of smooth functions
on a manifold $M$ valued in $\KK$, and $T_M$ refers to $T_M \otimes_\RR \KK$
unless stated otherwise.

In this paper, graded means $\ZZ$-graded. We write dg for differential graded.

We say that a graded ring $R$ is commutative if $xy = (-1)^{|x||y|} yx$
for all homogeneous $x,y \in R$.

When we use the notation $|x|=k$, we mean $x$ is a homogeneous element
in a graded $R$-module $V= \bigoplus_n V^n$ and the degree of $x$ is $k$,
i.e.\ $x \in V^k$. The notation $V[i]$ refers to the $R$-module $V$
with the shifted grading $(V[i])^k = V^{i+k}$.

Let $V,W,V',W'$ be graded modules over a graded ring $R$. We denote by $\Hom_R^k(V,W)$ the space of $R$-linear maps from $V$ to $W$ of degree $k$, and $\Hom_R(V,W) = \bigoplus_k \Hom_R^k(V,W)$. Here, we say $f:V \to W$ is $R$-linear of degree $|f|$ if $f(r \cdot x) = (-1)^{|r||f|}r \cdot f(x)$ for any homogeneous elements $r \in R$, $x \in V$. Note that $\Hom_R^0(V,W[i]) = \Hom_R^i(V,W)$. 
For $f\in \Hom_R^{|f|}(V,W)$ and $g\in \Hom_R^{|g|}(V',W')$, we denote by $f\otimes g$ the map $f\otimes g \in \Hom_R^{|f|+|g|}(V\otimes V', W \otimes W')$ satisfying $(f\otimes g)(v\otimes v') = (-1)^{|v||g|} f(v)\otimes g(v')$, $\forall\, v\in V,v'\in V'$.

Let $W$ be a graded module over a commutative graded ring $R$. The graded exterior algebra $\Lambda W$ generated by $W$ over $R$ is
$$
\Lambda W = \bigg(\bigoplus_{n=0}^\infty \overbrace{W \otimes_R \cdots \otimes_R W}^{n \text{ times}} \bigg) \Bigg/ \langle w_1 \otimes w_2 + (-1)^{|w_1||w_2|} w_2 \otimes w_1 \rangle,
$$
equipped with the product $\wedge$ induced by the tensor product and the degree assignment
$$
|w_1 \wedge \cdots \wedge w_n| = |w_1|+\cdots + |w_n|.
$$
We denote by $\Lambda^n W$ the image of
$\overbrace{W \otimes_R \cdots \otimes_R W}^{n \text{ times}}$
in $\Lambda W$ under the quotient map.
It is well-known that $\Lambda^n W$ and $\Syms{n}{W}[n]$ are isomorphic as graded modules. By the symbol $\Lambda^\bullet W$, we mean $\Lambda^\bullet W = \bigoplus_k (\Lambda^k W)[-k]$ which is isomorphic to $S(W[-1])$.

We use the triple $(W,\delta;h)$ of big space, coboundary map on big space and homotopy operator to represent the contraction data
$$
\begin{tikzcd}
(V , d)
\arrow[r, " \tau", shift left,hook] &  (W ,\delta)
\arrow[l, " \sigma", shift left,two heads] \arrow[loop, "h",out=12,in= -12,looseness = 3]
\end{tikzcd}.
$$ 
See Section~\ref{sec:CompareHLP} for how one generates the whole contraction data from the triple $(W,\delta;h)$.

\subsection*{Acknowledgments}
The author wishes to thank Ping Xu for suggesting this problem and for helpful discussions. 
The author is also grateful to Noriaki Ikeda, Camille Laurent-Gengoux and Mathieu Sti\'{e}non for fruitful discussions and useful comments.

\section{Preliminaries}

\subsection{Connections for a Lie algebroid}
Let $L$ be a Lie $\KK$-algebroid over a smooth manifold $M$ with the anchor map $\rho:L \to T_M $.  Let $E \to M$ be a $\KK$-vector bundle. An \textbf{$L$-connection} $\nabla$ on $E$ is a $\KK$-bilinear map 
$$
\nabla: \Gamma(L) \times \Gamma(E) \to \Gamma(E), \, (l,e) \mapsto \nabla_l e
$$
which satisfies the properties
\begin{gather*}
\nabla_{f\cdot l} e = f \cdot \nabla_l e, \\
\nabla_l(f \cdot e) = \rho(l)(f) \cdot e + f \nabla_l e,
\end{gather*}
for any $l \in \Gamma(L)$, $e \in \Gamma(E)$, and $f \in C^\infty (M)$. 
A \textbf{representation} of  $L$ on $E$ is a  \textbf{flat connection} $\nabla$ on $E$, i.e.\ a $L$-connection $\nabla: \Gamma(L) \times \Gamma(E) \to \Gamma(E)$ satisfying 
$$
\nabla_{l_1} \nabla_{l_2} e - \nabla_{l_2}\nabla_{l_1} e - \nabla_{[l_1,l_2]} e =0,
$$
for any $l_1,l_2 \in \Gamma(L)$ and $e \in \Gamma(E)$. A vector bundle equipped with a representation of the Lie algebroid $L$ is called an \textbf{$L$-module}.

Let $L$ be a Lie algebroid over a smooth manifold $M$. The \textbf{Chevalley--Eilenberg differential} is the linear map 
$$
d_L: \Gamma(\Lambda^k L\dual) \to \Gamma(\Lambda^{k+1}L\dual)
$$ 
defined by 
$$
(d_L\omega)(l_0, \cdots, l_k) = \sum_{i=0}^k (-1)^i \rho(l_i)\big(\omega(l_0, \cdots, \widehat{l_i}, \cdots, l_k)\big) + \sum_{i<j} (-1)^{i+j}\omega([l_i,l_j],l_0,\cdots, \widehat{l_i},\cdots,\widehat{l_j}, \cdots, l_k)
$$
which makes the exterior algebra $\bigoplus_{k=0}^\infty \Gamma(\Lambda^k L\dual)$ into a commutative dg algebra. Given an $L$-connection $\nabla$ on a vector bundle $E \to M$, the \textbf{covariant derivative} is the operator 
$$
d_L^\nabla:\Gamma(\Lambda^k L\dual \otimes E) \to \Gamma(\Lambda^{k+1} L\dual \otimes E)
$$
which maps a section $\omega \otimes e \in \Gamma(\Lambda^k L\dual \otimes E)$ to 
$$
d_L^\nabla(\omega\otimes e) = d_L(\omega) \otimes e + \sum_{i=1}^{\rk L} (\nu_i \wedge \omega)\otimes \nabla_{v_i} e,
$$
where $v_1,\cdots, v_{\rk L}$ is any local frame of the vector bundle $L$, and $\nu_1,\cdots, \nu_{\rk L}$ is its dual frame. The flatness of connection $\nabla$ is equivalent to that the covariant derivative $d_L^\nabla$ is a coboundary map: $d_L^\nabla \circ d_L^\nabla =0$.

\begin{example}\label{ex:Bott}
Let $(L,A)$ be a \textbf{Lie pair}, i.e.\ a pair of a Lie algebroid $L \to M$ and a Lie subalgebroid $A \to M$ of $L$. 
The {\bf Bott connection} of $A$ on the quotient bundle $B =L/A$ is the flat connection 
$$
\nabla^{\Bott}: \Gamma(A) \times \Gamma(B) \to \Gamma(B)
$$ 
defined by 
$$
\nabla^{\Bott}_a(\pb(l)) = \pb([a,l]), \qquad \forall \,a \in \Gamma(A), \, l \in \Gamma(L),
$$
where $\pb:L \onto B=L/A$ is the canonical projection. Its covariant derivative 
$$
d^{\Bott}: \Gamma(\Lambda^\bullet A\dual \otimes B) \to \Gamma(\Lambda^{\bullet +1} A\dual \otimes B)
$$
is called the {\bf Bott differential}. 
\end{example}

\subsection{Atiyah class and Todd class of a Lie pair}

A Lie pair $(L,A)$ consists of a Lie algebroid $L$ and a Lie subalgebroid $A$ of $L$ over a common base manifold $M$. 
The structure of Lie pairs arises from geometric problems naturally. 
A simple example is a pair of a Lie algebra and its Lie subalgebra. Such a pair is a Lie pair over a point. 
Another well-known example is from complex manifolds. If $X$ is a complex manifold, then the pair $(T_X \otimes \CC, T_X^{0,1})$ is a Lie pair.  
More generally, if $F$ is any Lie subalgebroid of the tangent bundle $T_M$, then the pair $(T_M,F)$ forms a Lie pair. Note that $F$ can be considered as the tangent bundle of a regular foliation.  
In Section~\ref{sec:Application}, we will also consider another type of Lie pairs which arise from $\frakg$-manifolds (i.e.\ manifolds with Lie algebra actions). 

Let $(L,A)$ be a Lie pair over a manifold $M$, and $B$ be the quotient vector bundle $B = L/A$. We have the short exact sequence 
$$
\begin{tikzcd}
0 \ar[r] & A \ar[r, hook, "\ia"] &   L \ar[r, two heads,"\pb"] &   B \ar[r] & 0
\end{tikzcd}
$$
of vector bundles. 
An $L$-connection $\nabla$ on $B$ is said to \textbf{extend the Bott connection}  if 
$$
\nabla_{\ia(a)} \big(\pb(l)\big) = \nabla^{\Bott}_a \big(\pb(l)\big) = \pb([\ia(a),l]),
$$
for any $a \in \Gamma(A)$, $l \in \Gamma(L)$.

Let $\nabla$ be an $L$-connection on $B$ extending the Bott connection. The curvature of $\nabla$ is the bundle map $R^\nabla: \Lambda^2 L \to \End B$ defined by 
$$
R^\nabla(l_1,l_2) = \nabla_{l_1} \nabla_{l_2} - \nabla_{l_2} \nabla_{l_1} - \nabla_{[l_1,l_2]}, \qquad \forall \, l_1,l_2 \in \Gamma(L).
$$
Since the Bott connection is flat, the restriction of $R^\nabla$ to $\Lambda^2 A$ vanishes. Thus, the curvature $R^\nabla$ induces a bundle map $\pairAtiyah\nabla:A \otimes B \to \End B$, 
$$
\pairAtiyah\nabla\big(a, \pb(l)\big) = R^\nabla(a,l) = \nabla_a \nabla_l - \nabla_l \nabla_a - \nabla_{[a,l]}, \qquad \forall \, a \in \Gamma(A),\, l \in \Gamma(L).
$$

\begin{remark}
An $L$-connection $\nabla$ on a vector bundle $E$ induces an $L$-connection on the vector bundle $E\dual \otimes \End(E) \cong \Hom(E\otimes E,E)$ via the equation 
$$
\nabla_l\big(\alpha(e_1\otimes e_2)\big) = \big(\nabla_l(\alpha)\big)(e_1\otimes e_2) + \alpha\big((\nabla_l e_1)\otimes e_2\big) + \alpha\big( e_1 \otimes (\nabla_l e_2)\big),
$$
for any $l\in\Gamma(L)$, $e_1,e_2 \in \Gamma(E)$, and $\alpha \in \Gamma(E\dual \otimes \End E)$. If the given $L$-connection on $E$ is flat, then the induced connection on $E\dual \otimes \End E$ is also flat. 
\end{remark}

\begin{proposition}[\cite{MR3439229}] 
The section $\pairAtiyah\nabla \in \Gamma(A\dual \otimes B\dual \otimes \End B)$ is a $1$-cocycle for the Lie algebroid $A$ with values in the $A$-module $B\dual \otimes \End B$. Furthermore, the cohomology class $\alpha_{L/A}\in H_{\CE}^1(A,B\dual \otimes \End B)$ of $ \pairAtiyah\nabla$ is independent of the choice of $L$-connection $\nabla$ extending the Bott connection. 
\end{proposition}

The cocycle $\pairAtiyah\nabla$ is called the \textbf{Atiyah cocycle} of the Lie pair $(L,A)$ associated with the $L$-connection $\nabla$. The induced cohomology class $\alpha_{L/A}\in H_{\CE}^1(A,B\dual \otimes \End B)$ is called the \textbf{Atiyah class} of the Lie pair $(L,A)$.

The \textbf{Todd cocycle} of a Lie pair $(L,A)$ associated with an $L$-connection $\nabla$ extending the Bott connection is the Chevalley--Eilenberg cocycle
$$
\Toddcocycle_{L/A}^\nabla = \det\left( \dfrac{\pairAtiyah\nabla}{1-e^{-\pairAtiyah\nabla}} \right) \in \bigoplus_{k=0}^\infty \Gamma(\Lambda^k A\dual \otimes \Lambda^k B\dual).
$$
The \textbf{Todd class} of a Lie pair $(L,A)$ is 
$$
\Todd_{L/A} = \det\left( \dfrac{\alpha_{L/A}}{1-e^{-\alpha_{L/A}}} \right) \in \bigoplus_{k=0}^\infty H_{\CE}^k( A, \Lambda^k B\dual).
$$
In the case of the Lie pair $(L,A) = (T_X \otimes \CC, T_X^{0,1})$ associated with a complex manifold $X$, the Atiyah class and the Todd class of the Lie pair are, respectively, the classical Atiyah class of $T_X$ and the classical Todd class of the complex manifold $X$.

\subsection{Atiyah class and Todd class of a dg Lie algebroid}

A \textbf{($\ZZ$-)graded manifold} $\cM$ is a pair $(M,\cO_\cM)$, where $M$ is a smooth manifold, and $\cO_\cM$ is a sheaf of $\ZZ$-graded commutative $\cO_M$-algebras over $M$ such that there exist (i) a $\ZZ$-graded vector space $V$, (ii) an open cover of $M$, and (iii) an isomorphism of sheaves of graded $\cO_U$-algebras $\cO_\cM|_U \cong \cO_U \otimes SV\dual$ for every open set $U$ of the cover. 
% there exists a $\ZZ$-graded vector space $V$ and an open covering of $M$, and for every $U$ in the covering family, we have $\cO_\cM|_U \cong C^\infty(U) \otimes SV\dual$.  
 A \textbf{dg manifold} $(\cM,Q)$ is a graded manifold $\cM$ endowed with a {\em homological vector field} $Q$, i.e.\ a derivation $Q$ of degree $+1$ of $C^\infty(\cM) = \cO_\cM(M)$ satisfying $[Q,Q]=0$. 
A \textbf{morphism} of dg manifolds $\phi:(M,\cO_\cM, Q) \to  (M',\cO_{\cM'}, Q')$ is a pair $\phi = (\underline{\phi}, \Phi)$, where $\underline{\phi}:M \to M'$ is a smooth map, and $\Phi:\cO_{\cM'} \to {\underline{\phi}}_{\ast}\cO_{\cM}$ is a morphism of sheaves of graded $\cO_{M'}$-algebras, such that $(\underline{\phi}_\ast Q) \circ \Phi = \Phi \circ Q'$. One also has the notion of morphisms of graded manifolds by regarding graded manifolds as dg manifolds with zero homological vector fields. 
See, for example, \cite{2023arXiv230710242B,MR2819233,MR2709144}.

\begin{example}
Let $A \to M$ be a vector bundle. Then $A[1]$ is a graded manifold, and its function algebra is $C^\infty(A[1]) \cong \Gamma(\Lambda^\bullet A\dual)$. If $A \to M$ is a Lie algebroid, then $A[1]$, together with the Chevalley--Eilenberg differential $d_A$, forms a dg manifold. According to Va\u{\i}ntrob \cite{MR1480150}, there is a bijection between the Lie algebroid structures on the vector bundle $A \to M$ and the homological vector fields on the $\ZZ$-graded manifold $A[1]$. 
\end{example}

A \textbf{(graded) vector bundle} of rank $\{k_i\}$ over a graded manifold $\cM$ is a graded manifold $\cE$ and a {\em surjection} $\pi=(\underline \pi,\Pi):\cE \to \cM$ (i.e.\ $\underline\pi:E \to M$ is surjective and $\Pi:\cO_\cM \to \underline\pi_\ast\cO_{\cE}$ is injective) equipped with an atlas of local trivializations $\cE|_{\underline\pi\inv(U)} \cong \cM|_U \times (\bigoplus_i \RR^{k_i}[-i])$ such that the transition map between any two local trivializations is linear in the fiber coordinates. Given a graded vector bundle $\pi:\cE \to \cM$, one can shift the degrees of fibers and obtain another graded vector bundle  $\pi[j]:\cE[j] \to \cM$. We will denote this degree-shifting functor by $[j]_\cM$ when the base graded manifold is ambiguous. The section space $\Gamma(\cE)$ of $\pi:\cE\to \cM$ is defined to be $\bigoplus_{j \in \ZZ}\Gamma^j(\cE)$, where $\Gamma^j(\cM)$ consists of morphisms $s:\cM \to \cE[j]$ such that $\pi[j]\circ s = \id_\cM$.
 
A graded vector bundle $\pi:\cE \to \cM$ is called a {\bf dg vector bundle} if $\cE$ and $\cM$ are both dg manifolds, $\pi$ is a morphism of dg manifolds, and the dg structure is compatible with the vector bundle structure in the following sense: 
the subset $\Gamma(\cE\dual)$ of $C^\infty(\cE)$, consisting of the fiberwise linear functions on $\cE$, remains stable under the homological vector field $Q_\cE \in \XX(\cE)$. 
It is well-known that the global sections of a dg vector bundle form a dg module over the dg algebra of functions on the base dg manifold.  See \cite{MR2709144,MR2534186,MR3319134} for further details.   
A more general concept of ``dg fiber bundles'' (also known as $Q$-bundles) and their relationship with gauge fields can be found in \cite{MR3293862}.

A {\bf graded Lie algebroid} $\cL \to \cM$ is a Lie algebroid object in the category of graded manifolds. More explicitly, it is a graded vector bundle $\cL \to \cM$ together with a degree-zero bundle map $\rho: \cL \to T_\cM$ (the {\bf anchor}) and a degree-zero Lie bracket $[\argument, \argument]: \Gamma(\cL) \times \Gamma(\cL) \to \Gamma(\cL)$ such that 
$$
[X,fY] = \rho(X)(f) \cdot Y + (-1)^{|X||f|}f [X,Y],
$$
for any $X,Y \in \Gamma(\cL)$ and $f \in C^\infty(\cM)$. 
According to a well-known theorem of Va\u{\i}ntrob \cite{MR1480150}, the Chevalley--Eilenberg differential 
$$
d_\cL: \Gamma(\Lambda^\bullet \cL\dual) \to \Gamma(\Lambda^{\bullet + 1} \cL\dual)
$$
of the graded Lie algebroid $\cL \to \cM$ can be viewed as a homological vector field on $\cL[1]$ so that $(\cL[1], d_\cL)$ is a dg manifold. 

Assume $\cL \to \cM$ is a dg vector bundle. Since the homological vector field $Q_\cL \in \XX(\cL)$ preserves the fiberwise linear functions $\Gamma(\cL\dual)$ on $\cL$, it induces a homological vector field $\tilde\cQ_\cL$ on $\cL[1]$. The following definition is due to Mehta \cite{MR2534186}.

\begin{definition}
A \textbf{dg Lie algebroid} consists of a dg vector bundle $\cL \to \cM$ equipped with a pair of homological vector fields $Q_\cL$ and $Q_\cM$ on $\cL$ and $\cM$, respectively, and a graded Lie algebroid structure on the vector bundle $\cL \to \cM$ such that the dg and the graded Lie algebroid structures are compatible in the sense that the Chevalley--Eilenberg differential $d_\cL$ of the graded Lie algebroid structure and the homological vector field $\tilde\cQ_\cL$ on $\cL[1]$ induced by $\cQ_\cL$ --- two derivations of the graded algebra $C^\infty(\cL[1]) \cong \Gamma(\Lambda^\bullet \cL\dual)$ --- commute:
$$
[d_\cL,\tilde\cQ_\cL] = 0.
$$
\end{definition}

%As in the ordinary case, a graded Lie algebroid structure on $\cL$ corresponds to a homological vector field $d_\cL$ on $\cL[1]$. 
%A dg vector bundle structure on $\cL$ also induces a homological vector field $Q$ on $\cL[1]$. A \textbf{dg Lie algebroid} $\cL \to \cM$ is  a graded Lie algebroid together with a dg vector bundle structure on it such that the induced homological vector fields $d_\cL$ and $Q$ on $\cL[1]$ are compatible in the sense that $[d_\cL,Q] =0$. 

Note that, if $\cL \to \cM$ is a dg Lie algebroid, then the function algebra $C^\infty(\cL[1])$ with the operators $d_\cL$ and $\tilde\cQ_\cL$ form a double complex. See \cite[Section~4]{MR2534186}. 

In the following, we describe a few ways to construct dg Lie algebroids. 

\begin{example}
A fundamental example of dg Lie algebroid is the tangent bundle of a dg manifold. If $(\cM,Q)$ is a dg manifold, then the Lie derivative $L_Q$ defines a dg structure on its tangent bundle $T_\cM \to \cM$. This dg structure and the standard Lie bracket of vector fields form a dg Lie algebroid structure on $T_\cM$.
\end{example}

\begin{example}
For more sophisticated examples, one can consider double Lie algebroids \cite{1998math......8081M,MR1650045}.  According to \cite{MR2534186} (also see \cite{MR2971727}), from a double Lie algebroid 
$$
\begin{tikzcd}
D \ar[r] \ar[d] & B \ar[d] \\
A \ar[r] & M,
\end{tikzcd}
$$
one can construct two dg Lie algebroids: $D[1]_A \to B[1]$ and $D[1]_B \to A[1]$. These two dg Lie algebroids can be considered to be dual to each other. 
\end{example}

\begin{example}
Let $A \to M$ be a Lie algebroid. The graded vector bundle $T_A[1]_A \to T_M[1]$ is naturally a dg Lie algebroid \cite[Section~5.2]{MR2534186}. If $A=\frakg$ is a Lie algebra, then the double complex associated with the dg Lie algebroid $T_\frakg[1]_\frakg \to T_{\text{pt}}[1]$ is isomorphic to the Weil algebra $W(\frakg)=\Lambda^\bullet \frakg\dual \otimes S (\frakg\dual[-2])$. See \cite[Section~5.3]{MR2534186}.
% $C^\infty((T_\frakg[1]_\frakg)[1]_{\text{pt}}) \cong C^\infty(\frakg[1]\oplus \frakg[2]) \cong \Lambda^\bullet \frakg\dual \otimes S (\frakg\dual[-2]) =:W(\frakg)$ is the Weil algebra \cite[Section~5.3]{MR2534186}. 
If $A = \frakg  \ltimes M \to M$ is the action Lie algebroid, then the double complex associated with the dg Lie algebroid $T_A[1]_A \to T_M[1]$ is isomorphic to the BRST model of equivariant cohomology \cite[Example~5.12]{MR2534186}. 
\end{example}

\begin{example}
Let $(L,A)$ be a Lie pair over a manifold $M$, and $B=L/A$ be the quotient vector bundle. It is known \cite{MR4150934,MR3964152} that $\cM = L[1]\oplus B$ is a (formal) dg manifold, called a {\em Fedosov dg manifold}. Furthermore, the pullback $\cF \to \cM$ of $B\to M$ via the canonical projection $\cM \to M$ is a dg Lie subalgebroid of the tangent dg Lie algebroid $T_\cM \to \cM$. This dg Lie algebroid $\cF \to \cM$ is called a {\em Fedosov dg Lie algebroid} associated with a Lie pair \cite[Appendix~A]{MR3964152}.  See \cite{MR3754617} for a parallel construction: Fedosov dg Lie algebroids associated with dg manifolds. The construction of Fedosov dg Lie algebroid was adapted from Fedosov's iteration techniques in deformation quantization  \cite{MR1293654}. Fedosov dg Lie algebroids are important in the study of Kontsevich-type formality morphisms \cite{MR3964152,MR3754617,MR3650387,MR2102846,MR2199629} and Atiyah classes \cite{MR3724780,MR3964152}.
\end{example}

We will need dg Lie algebroids of the following type. 

\begin{proposition}\label{prop:InnerDGLieAbd}
Let $\cL \to \cM$ be a graded Lie algebroid. A section $s$ of degree $+1$ of $\cL \to \cM$ satisfying $[s,s]=0$ induces a dg Lie algebroid structure on $\cL$ with its induced differential on $\Gamma(\cL)$ being $\cQ = [s, \argument]: \Gamma(\cL) \to \Gamma(\cL)$. 
\end{proposition}

A proof of Proposition~\ref{prop:InnerDGLieAbd} can be found in \cite{2022arXiv221016769S}.

Let $\cE \to \cM$ be a dg vector bundle, and let $\cL \to \cM$ be a dg Lie algebroid with anchor $\rho:\cL \to T_\cM$. We denote both the induced differentials on $\Gamma(\cE)$ and $\Gamma(\cL)$ by $\cQ$.  An \textbf{$\cL$-connection} on $\cE \to \cM$ is a degree-preserving map $\nabla: \Gamma(\cL) \otimes_\KK \Gamma(\cE) \to \Gamma(\cE)$ such that 
\begin{gather*}
\nabla_{f \cdot \lambda} \epsilon = f\cdot \nabla_\lambda \epsilon, \\
\nabla_\lambda (f\cdot \epsilon) = \rho(\lambda)(f) \cdot \epsilon + (-1)^{|\lambda|} f \cdot \nabla_\lambda \epsilon,
\end{gather*}
for $f \in C^\infty(\cM)$, $\lambda \in \Gamma(\cL)$, and $\epsilon \in \Gamma(\cE)$. 
Given a dg vector bundle $\cE \to \cM$  and an $\cL$-connection $\nabla$ on it, we consider the bundle map $\At^\nabla_\cE: \cL \otimes \cE \to \cE$ defined by 
$$
\At^\nabla_\cE(\lambda,\epsilon) = \cQ(\nabla_\lambda \epsilon) - \nabla_{\cQ(\lambda)} \epsilon - (-1)^{|\lambda|} \nabla_\lambda \big(\cQ(\epsilon)\big), \qquad \forall \, \lambda \in \Gamma(\cL), \, \epsilon \in \Gamma(\cE).
$$

\begin{proposition}[\cite{MR3319134}]
The bundle map $\At^\nabla_\cE$ is a degree $+1$ section of $\cL\dual \otimes \End \cE$   satisfying the cocycle equation: $\cQ(\At^\nabla_\cE) =0$. The cohomology class $\alpha_\cE \in H^1\big(\Gamma(\cL\dual \otimes \End \cE), \cQ\big)$ of $\At^\nabla_\cE$ is independent of the choice of the $\cL$-connection $\nabla$. 
\end{proposition}

The cocycle $\At^\nabla_\cE$ is called the \textbf{Atiyah cocycle} associated with the $\cL$-connection $\nabla$. The induced cohomology class $\alpha_{\cE} = [\At^\nabla_\cE] \in H^1\big(\Gamma(\cL\dual \otimes \End \cE), \cQ\big)$ is called the \textbf{Atiyah class} of the dg vector bundle $\cE \to \cM$ relative to the dg Lie algebroid $\cL \to \cM$. 
If $\cE = \cL$, we say that $\alpha_\cL$ is the Atiyah class of the dg Lie algebroid $\cL$. If $\cE = \cL=T_\cM$, we say that $\alpha_{T_\cM}$ is the Atiyah class of the dg manifold $\cM$.

The \textbf{Todd cocycle} of a dg vector bundle $\cE$ associated with an $\cL$-connection $\nabla$ is the $\cQ$-cocycle
$$
\Toddcocycle_{\cE}^\nabla = \Ber\left( \dfrac{\At^\nabla_\cE}{1-e^{-\At^\nabla_\cE}} \right) \in \prod_{k=0}^\infty \big(\Gamma(\Lambda^k \cL\dual)\big)^k,
$$
and the \textbf{Todd class} of a dg vector bundle $\cE$ relative to a dg Lie algebroid $\cL$ is 
$$
\Todd_{\cE} = \Ber\left( \dfrac{\alpha_{\cE}}{1-e^{-\alpha_{\cE}}} \right) \in \prod_{k=0}^\infty H^k(\Gamma(\Lambda^k \cL\dual), \cQ),
$$
where $\Ber$ denotes the Berezinian. 
It is well known that $\Todd_\cE$ can be expressed in terms of scalar Atiyah classes $\frac{1}{k!}(\frac{i}{2\pi})^k \str \alpha_\cE^k \in H^k(\Gamma(\Lambda^k \cL\dual),\cQ)$. Here $\str: \Lambda \cL\dual \otimes \End \cE \to \Lambda \cL\dual$ denotes the supertrace.

\section{Dg Lie algebroid associated with a Lie pair}

Let $(L,A)$ be a Lie pair over a manifold $M$, and let  $\pi_L:L \to M$ be the bundle projection. We denote by $\pi:A[1] \onto M$ the bundle projection of the shifted vector bundle $A[1]$. 
In \cite{2022arXiv221016769S}, Sti\'{e}non, Vitagliano and Xu investigated the pullback Lie algebroid $\pi^!L$ of $L \to M$ along $\pi:A[1] \to M$. They proved that $\pi^!L$ is equipped a dg Lie algebroid structure and constructed a contraction data 
$$
\begin{tikzcd}
\big(\Gamma(M,\Lambda^\bullet A\dual \otimes B) , d^{\Bott}\big)
\arrow[r, " \tau", shift left,hook] &  \big(\Gamma(A[1],\pi^!L) ,\cQ\big)
\arrow[l, " \pb", shift left,two heads] \arrow[loop, " \hpt",out=7,in= -7,looseness = 3]
\end{tikzcd},
$$ 
where $B=L/A$, $d^{\Bott}$ is the Bott differential, and $\cQ$ is induced by the dg Lie algebroid structure of $\pi^!L$.  Sti\'{e}non--Vitagliano--Xu's method is computational and heavily based on explicit formulas. Here, we give an alternative construction of this contraction by the homological perturbation lemma.

\subsection{The pullback Lie algebroid $\pi^! L$}\label{sec:pi!L}

The {\bf pullback Lie algebroid} (see \cite[Section~4.2]{MR2157566}) of $L$ via $\pi:A[1]\onto M$ is  the vector bundle 
$$
\pi^! L := T_{A[1]} \times_{T_M} L
$$
over $A[1]$ with the graded Lie algebroid structure described in the next paragraph. Note that (i) the two maps for defining the fiber product are the tangent map $\pi_\ast:T_{A[1]} \to T_M$ of  $\pi$ and the anchor $\anchorL:L \to T_M$ of $L$, (ii) $\pi^! L$ is a vector bundle over $A[1]$ whose bundle projection is the composition $ T_{A[1]} \times_{T_M} L \onto T_{A[1]} \onto A[1]$, and (iii) a general section of $\pi^! L \to A[1]$ is of the form 
\begin{equation}\label{eq:Sec_piL}
(X,v), \qquad \forall X \in \XX(A[1]), \, v \in \Gamma(\pi^*L),
\end{equation}
satisfying the condition
\begin{equation}\label{eq:SecCondition_piL}
\pi_\ast X = \anchorL \circ v: A[1] \to T_M.
\end{equation}
In the equation \eqref{eq:SecCondition_piL}, a vector field $X$ on $A[1]$ is regarded as a map $X:A[1] \to T_{A[1]}$, and a section 
$$
v \in \Gamma(\pi^* L) \cong C^\infty(A[1])\otimes_{C^\infty(M)} \Gamma(L) 
$$ 
is identified with a smooth map (not necessarily linear) $v:A[1] \to L$ such that $\pi_L \circ v = \pi$. Also note that the space $\Gamma(\pi^* L)$ is generated by 
$$
l \circ \pi, \qquad l \in \Gamma(L),
$$
as a $C^\infty(A[1])$-module.

The pullback Lie algebroid $\pi^!L$ is equipped with the anchor  
$$
\anchor: \pi^!L = T_{A[1]} \times_{T_M} L \onto T_{A[1]}, \, \anchor(X,v) = X.
$$
The Lie bracket on $\Gamma(\pi^!L)$ is characterized by the equation 
$$
[(X,l\circ \pi),(X', l' \circ \pi)]:= ([X,X'], [l,l']\circ \pi),
$$
for $X,X' \in \XX(A[1])$ and $l, l' \in \Gamma(L)$. 
More explicitly, for $l_i, l_j' \in \Gamma(L)$ and $f_i, f_j' \in C^\infty(A[1])$, we have 
\begin{multline*}
\Big[(X, \sum_i f_i \otimes l_i\circ \pi), \, (X', \sum_j f_j' \otimes l_j' \circ \pi)\Big] \\ = \Big([X,X'],\, 
 \sum_j X(f_j') \otimes l_j'  \circ \pi   
-  \sum_i (-1)^{|X'||f_i|} X'(f_i) \otimes l_i\circ \pi + \sum_{i,j} (f_if_j')\otimes [l_i,l_j'] \circ \pi \Big)
\end{multline*}
in $\Gamma(\pi^!L)$.

\subsection{Contractions induced by splittings}

By choosing a connection of $A$, one can decompose $T_A$ as the direct sum of a vertical subbundle $V\cong A \times_M A$ and a horizontal subbundle $H \cong A \times_M T_M $. Such a connection induces an isomorphism
\begin{equation}\label{eq:ConnectionDecomp}
T_{A[1]} \cong A[1] \times_M (A[1] \oplus T_M).
\end{equation}
Thus, we have 
\begin{equation}
\pi^! L = T_{A[1]} \times_{T_M} L  \cong A[1] \times_M (A[1] \oplus  L) \cong \pi^*(A[1]) \oplus  \pi^*(L),
\end{equation}
where the last direct sum is a direct sum of vector bundles over $A[1]$. As a consequence, we have
\begin{equation}\label{eq:TrivializedPullbackAbd}
\Gamma(\pi^! L) \cong \Gamma(\pi^* A[1])  \oplus \Gamma(\pi^* L)
\end{equation}
as $C^\infty(A[1])$-modules.

\begin{remark}\label{rmk:TwoExpressionpi!L}
In \eqref{eq:Sec_piL}, we describe a general section of $\pi^! L \to A[1]$ by a pair $(X,v)$ of a vector field $X \in \XX(A[1])$ and a section $v \in \Gamma(\pi^* L)$ which satisfies \eqref{eq:SecCondition_piL}. By choosing a connection, one has a horizontal subbundle $H$ in $T_{A[1]}$ which is isomorphic to $\pi^*T_M$ via $\pi_\ast$.  Let $\psi :\pi^* T_M  \to H $ be the inverse of $\pi_\ast$. Then an element $(x,v) \in \Gamma(\pi^* A[1])  \oplus \Gamma(\pi^* L)$ corresponds to the pair $\big(x+ \psi\big((\id\otimes\anchorL)(v)\big), v\big)\in \Gamma(\pi^!L) $, where $x \in \Gamma(\pi^* A[1])$ is identified with its associated vertical vector field on $A[1]$, $v \in \Gamma(\pi^*L) \cong C^\infty(A[1]) \otimes \Gamma(L)$, and $\id \otimes \anchorL:C^\infty(A[1]) \otimes \Gamma(L) \to C^\infty(A[1]) \otimes \XX(M) \cong \Gamma(\pi^* T_M)$.
\end{remark}

Note that if $W$ is a vector space, and if $V$ is a subspace of $W$,
then the two-term complex $0 \to V \into W \to 0$ is homotopy equivalent
to the quotient space $W/V$. A choice of splitting of
\begin{equation}\label{saturn}
0 \to V \into W \onto W/V \to 0
\end{equation}
induces a homotopy inverse $i:W/V \to W$ of the quotient map $W \onto W/V$
and a homotopy operator $p:W \to V$
\begin{equation}\label{eq:VecSpContraction}
\begin{tikzcd}
V \ar[r,hook] \ar[d,shift right] &\ar[l,bend right,"p"'] W \ar[d,shift right,two heads] \\
0 \ar[u,shift right] \ar[r] & W/V \ar[u,shift right,"i"'].
\end{tikzcd}
\end{equation}
In the following, we apply this observation fiberwisely to the pullback bundles $\pi^*A \subset \pi^*L$ with proper degree-shifting, and we obtain a homotopy equivalence between $\Gamma(\pi^!L)\cong \Gamma(\pi^*A[1]) \oplus \Gamma(\pi^*L)$ with the differential induced by the inclusion map and $\Gamma\big(\pi^*(L/A)\big) \cong \Gamma\big(\Lambda^\bullet A\dual \otimes (L/A)\big)$ with the zero differential.

Let $B \to M$ be the quotient vector bundle $L/A$. 
Let $\ia: A \into L$ be the inclusion map, and $\pb:L \onto B$ be the projection map. For simplicity, we also denote the induced inclusion $\ia: \pi^* A \into \pi^* L$ and projection $\pb: \pi^*L \onto \pi^*B$ by the same notations. 
Let $\pa   :\pi^* L \to \pi^* A$ be a splitting of the short exact sequence
\begin{equation}\label{eq:splitting}
\begin{tikzcd}
0 \ar[r] & \pi^* A \ar[r, hook, "\ia"'] & \ar[l, bend right,dashed,"\pa"'] \pi^*L \ar[r, two heads,"\pb"'] & \ar[l, bend right,dashed,"\ib"'] \pi^*B \ar[r] & 0,
\end{tikzcd}
\end{equation}
and let $\ib: \pi^* B \to \pi^* L$ be the inclusion map such that $\pb \ib = \id$ and $\ia \pa + \ib \pb = \id$. 

Let $\incl: \Gamma(\pi^! L) \to \Gamma(\pi^! L)$ be the $C^\infty(A[1])$-linear operator
$$
\incl(x,v) := \big(0, \ia(\pshift(x))\big)
$$
of degree $+1$, and let $\hpt:\Gamma(\pi^! L) \to \Gamma(\pi^! L)$ be the  $C^\infty(A[1])$-linear operator
$$
\hpt(x,v) := \big(\nshift (\pa(v)), 0\big)
$$
of degree $-1$, 
where $(x,v) \in \Gamma(\pi^* A[1])  \oplus \Gamma(\pi^* L)\cong \Gamma(\pi^! L)$, $\pshift: \Gamma(\pi^* A[1]) \to \Gamma(\pi^* A)$ is the degree-shifting map of degree $+1$, and $\nshift: \Gamma(\pi^* A) \to \Gamma(\pi^* A[1])$ is the degree-shifting map of degree $-1$. Since $|\incl| =1$ and $|\hpt| = -1$, one has  
$$
\incl(f \cdot \lambda) = (-1)^{|f|} f \cdot \incl(\lambda), \qquad \hpt(f \cdot \lambda) = (-1)^{|f|} f \cdot \hpt(\lambda),
$$
for $f \in C^\infty(A[1])$ and $\lambda \in \Gamma(\pi^! L)$.
Also, note that the pairs $( \Gamma(\pi^! L), \incl )$ and  $( \Gamma(\pi^* B), 0)$ are  dg modules over $(C^\infty(A[1]), 0)$, and the projection map $\pb:( \Gamma(\pi^! L), \incl ) \to ( \Gamma(\pi^* B), 0)$ forms a homotopy equivalence with the homotopy inverse $\ib$ and the homotopy operator $\hpt$.
$$
\begin{tikzcd}
 \Gamma(\pi^*A[1])\ar[d, shift right] \ar[r, hook, "\incl"'] & \ar[l, bend right, "\hpt"'] \Gamma(\pi^*L) \ar[d, two heads, shift right,"\pb"'] \\
 0 \ar[u, shift right] \ar[r] & \Gamma(\pi^*B) \ar[u, shift right,"\ib"']
\end{tikzcd}
$$

The following proposition is straightforward.

\begin{proposition}\label{prop:BasicContractionLiePair}
The diagram 
$$
\begin{tikzcd}
\big(\Gamma(\pi^* B),0\big)
\arrow[r, " \ib", shift left,hook] &  \big(\Gamma(\pi^! L), \incl\big)
\arrow[l, " \pb", shift left,two heads] \arrow[loop, "\hpt",out=12,in= -12,looseness = 3]
\end{tikzcd}
$$
forms a contraction data over $(C^\infty(A[1]),0)$. 
\end{proposition}

\subsection{Dg Lie algebroid structures on $\pi^!L$}\label{sec:DGLieAbd}

Let $d_A:C^\infty(A[1]) \to C^\infty(A[1])$ be the Chevalley--Eilenberg differential. The following lemma was proved in \cite{2022arXiv221016769S}.

\begin{lemma}[\cite{2022arXiv221016769S}]\label{lem:SecFromLieAbdMor}
If $\phi:A \to L$ is a Lie algebroid morphism, then the pair 
$$
s_\phi:=(d_A, \tilde\phi) 
$$ 
defines a section of $\pi^!L \to A[1]$, following the description of sections as given in \eqref{eq:Sec_piL}, where $\tilde\phi:= \phi \circ \pshift: A[1] \to A \to L$. This section satisfies the property 
$$
[s_\phi,s_\phi]=0.
$$
\end{lemma}

In particular, the inclusion map $\ia:A \into L$ induces a section $s_{\ia} \in \Gamma(\pi^!L)$ with the property $[s_{\ia},s_{\ia}] =0$. We denote  
$$
\cQ:=[s_{\ia},\argument]:\Gamma(\pi^!L) \to \Gamma(\pi^!L)
$$
which is an operator of degree $+1$ such that $\cQ^2=0$. By Proposition~\ref{prop:InnerDGLieAbd}, we have 

\begin{proposition}[\cite{2022arXiv221016769S}]\label{prop:DGLieAbdFromLiePair}
The pullback Lie algebroid $\pi^!L$ is a dg Lie algebroid over the dg manifold $(A[1],d_A)$ whose global sections are equipped with the differential $\cQ$.
\end{proposition}

\begin{remark}[Local formulas]\label{rmk:LocFormula}
Here, we choose a trivialization \eqref{eq:TrivializedPullbackAbd} and a splitting \eqref{eq:splitting}. Let $x^1, \cdots, x^n$ be a local coordinate system on $M$, and let $e_1, \cdots, e_r$ be a local frame of $A \to M$ that extends to a local frame $e_1, \cdots e_{r+r'}$ of $L \to M$ so that $e_j = \ib\pb(e_j)$ for $j=r+1, \cdots, r+r'$. 
We also denote the induced local frame of $\pi^*L$  by the same notations, $e_1, \cdots, e_{r+r'}$.

In addition, let $\eta^1, \cdots, \eta^r$ be the corresponding local frame of $(A[1])\dual \to M$, and let $\peta 1, \cdots, \peta r \in \Gamma(T_{A[1]}^{\ver}) \cong\Gamma(\pi^* A[1])$ be the corresponding local vertical vector fields. 
More explicitly, we choose $\eta^i$ and $\peta  j$ so that 
$$
\peta j (\eta^i) = \delta_j^i, \qquad \text{and} \qquad \hpt(e_j) = \peta j,
$$
for $i,j = 1, \cdots, r$. 

Let $\rho_i = \anchorL(e_i) = \sum_{j=1}^n \rho_i^j(x) \px{j}$ and $[e_i,e_j] = \sum_{k=1}^{r+r'} c_{ij}^k (x) e_k$.  
We have the local formula for the Chevalley--Eilenberg differential: 
\begin{equation}\label{eq:CE_local}
d_A = \sum_{i=1}^r \sum_{j=1}^n \rho_i^j \eta^i \px{j} - \frac{1}{2} \sum_{i,j,k =1}^r c_{ij}^k \eta^i \eta^j  \peta{k},
\end{equation}
where $\px j$ are regarded as horizontal vector fields on $A[1]$ via \eqref{eq:ConnectionDecomp}. 
Furthermore, in $\Gamma(\pi^* A[1])  \oplus \Gamma(\pi^*L)$, we have 
\begin{align}
\cQ\Big(\peta{l}\Big) & = \sum_{i,k =1}^r c_{il}^k \eta^i \peta k + e_l   , \label{eq:cQ_A[1]} \\
\cQ (e_l) &  =  \frac{1}{2}  \sum_{i,j,k=1}^r \rho_l(c_{ij}^k) \eta^i \eta^j \peta k  +  \sum_{k=1}^{r+r'} \sum_{i=1}^r \eta^i c_{il}^k e_k  . \label{eq:cQ_L}
\end{align}
\end{remark}

According to Proposition~\ref{prop:BasicContractionLiePair},   we have the contraction $(\Gamma(\pi^! L), \incl; \hpt)$. Let 
$$
\partial = \cQ - \incl, 
$$
and 
$$
F^q(\pi^! L) = \begin{cases}
\Gamma(\pi^! L) \cong \Gamma(\pi^* A[1])\oplus \Gamma(\pi^* L), & \text{ if } q \leq 0, \\
\Gamma(\pi^*A[1]), & \text{ if } q =1, \\
0, & \text{ if } q \geq 2.
\end{cases}
$$
It is clear that $F$ is an exhaustive complete filtration.

\begin{lemma}
The operator $\partial$ is a perturbation of $(\Gamma(\pi^! L), \incl)$ over $\KK$, and it satisfies the property 
$$
(\partial \hpt)\big(F^q(\pi^! L)\big) \subset F^{q+1}(\pi^! L)
$$ 
for all $q$. 
\end{lemma}
\begin{proof}
Since $\incl$ and $\hpt$ are $C^\infty(A[1])$-linear, the derivation property of $\cQ$ implies that 
$$
(\partial \hpt)(f \cdot \lambda) = (-1)^{|f|} d_A(f) \cdot \hpt(\lambda)+   f \cdot (\partial \hpt)(\lambda)
$$
for $f \in C^\infty(A[1])$, $\lambda \in \Gamma(\pi^! L)$. Due to this algebraic property of $\partial \hpt$, it suffices to show that $(\partial \hpt)\Big(\peta l\Big) \in F^2(\pi^! L)$ and $(\partial \hpt)(e_l) \in F^1(\pi^! L)$: 
\begin{align*}
(\partial \hpt)\Big(\peta l\Big) & = 0 \in F^2(\pi^! L), \\
(\partial \hpt)(e_l) & = \partial \Big(\peta l\Big)  =  \sum_{i,k =1}^r c_{il}^k \eta^i \peta k + e_l   -  e_l   \\
& =  \sum_{i,k =1}^r c_{il}^k \eta^i \peta k \in \Gamma(\pi^* A[1]) = F^1(\pi^! L).
\end{align*}
This completes the proof. 
\end{proof}

By Theorem~\ref{thm:HPLCptForm} and Corollary~\ref{cor:FormulaPertOp}, we have the following 

\begin{theorem}\label{prop:ContPullbackLieAbd}
The operator $\partial = \cQ -\incl$ is a small perturbation of the contraction $\big(\Gamma(\pi^! L), \incl; \hpt \big)$ over $\KK$. The perturbed contraction 
\begin{equation}\label{eq:piL_contraction}
\begin{tikzcd}
\big(\Gamma(\pi^* B),d^{\Bott}\big)
\arrow[r, " \tau ", shift left,hook] &  \big(\Gamma(\pi^! L), \cQ \big)
\arrow[l, " \pb", shift left,two heads] \arrow[loop, "\hpt",out=12,in= -12,looseness = 3]
\end{tikzcd},
\end{equation}
forms a contraction data over $(C^\infty(A[1]),d_A)$. Here, the coboundary operator $d^{\Bott}$ is the Bott differential, $\cQ =[s_{\ia},\argument]$, and 
\begin{equation}\label{eq:tauMain}
\tau   = \ib - \hpt \partial \, \ib = \ib - \hpt \cQ \, \ib: \Gamma(\pi^* B) \into \Gamma(\pi^! L).
\end{equation}
\end{theorem}
\begin{proof}
Observe that 
$$
\hpt \partial\, \hpt =0, \quad \partial \ib = \cQ \ib, \quad \pb \partial = \pb \cQ, \quad \pb \cQ \hpt =0.
$$
Thus, by Corollary~\ref{cor:FormulaPertOp}, the perturbed operators are 
\begin{align*}
(\hpt)_\partial & = \sum_{k=0}^\infty \hpt(-\partial \hpt)^k = \hpt, \\
(\ib)_\partial & = \sum_{k=0}^\infty (- \hpt \partial)^k \ib   = \ib - \hpt  \cQ   \ib,  \\
(\pb)_\partial & =  \sum_{k=0}^\infty \pb (-\partial \hpt)^k = \pb - \pb \cQ  \hpt = \pb, \\
\delta_\partial  & =  \sum_{k=0}^\infty \pb  \partial(- \hpt \partial)^k \ib = \pb \cQ \ib - \pb \cQ \hpt \cQ   \ib= \pb \cQ \ib, 
\end{align*}
where $\delta_\partial:\Gamma(\pi^* B) \to \Gamma(\pi^* B)$ is the perturbed differential defined as in Definition~\ref{def:PertOp}. 
By $\hpt \ib =0$, it can be easily shown that the perturbed inclusion $\tau = (\ib)_\partial= \ib - \hpt  \cQ   \ib$ is $C^\infty(A[1])$-linear.

Let $\bar e_l = \pb(e_l)$, $l=r+1, \cdots, r+r'$. 
Since  
$$
\pair{e_i}{d^{\Bott}(\bar e_l)} = \nabla^{\Bott}_{e_i} \bar e_l = \sum_{k=r+1}^{r+r'} c_{il}^k \bar e_k, \qquad \forall \, i = 1, \cdots, r, \, \forall \, l = r+1, \cdots, r+r',
$$
we have 
$$
d^{\Bott}(\bar e_l) = \sum_{i=1}^r \sum_{k=r+1}^{r+r'} \eta^i c_{il}^k \bar e_k.
$$
Furthermore, 
\begin{align*}
\pb \cQ   \ib(\bar e_l) &= \pb \cQ  (e_l) \\
& = \pb \Big(\frac{1}{2} \sum_{s=1}^{n} \sum_{i,j,k=1}^r \rho_l(c_{ij}^k) \eta^i \eta^j \peta k + \sum_{k=1}^{r+r'} \sum_{i=1}^r \eta^i c_{il}^k e_k\Big) \\
& = \sum_{k=r+1}^{r+r'} \sum_{i=1}^r \eta^i c_{il}^k \bar e_k \\
& = d^{\Bott}(\bar e_l),
\end{align*}
for $l = r+1, \cdots, r+r'$. Since both the operators $\pb \cQ   \ib$ and $d^{\Bott}$ satisfy the equation
$$
D(f \cdot b) = d_A(f) \cdot b + (-1)^{|f|} f \cdot  D(b),
$$
for $f \in C^\infty(A[1])$, $b \in \Gamma(B)$, $D = \hpt \cQ   \ib$ or 
$d^{\Bott}$, we conclude that $\delta_\partial   =  d^{\Bott}$.
\end{proof}

The contraction \eqref{eq:piL_contraction} coincides with Sti\'{e}non--Vitagliano--Xu's contraction in \cite{2022arXiv221016769S}.

\section{Two Atiyah classes associated with a Lie pair}

Let $(L,A)$ be a Lie pair over a manifold $M$, and let $\nabla$ be an $L$-connection on $B=L/A$ extending the Bott connection. We denote by $\pairAtiyah\nabla \in \Gamma(A\dual \otimes B\dual \otimes \End B)$ the Atiyah cocycle of the Lie pair $(L,A)$ associated with the connection $\nabla$. Let $\cL$ be the dg Lie algebroid $\pi^!L \to A[1]$, as described in Proposition~\ref{prop:DGLieAbdFromLiePair}. The contraction \eqref{eq:piL_contraction} induces a contraction $\big( \Gamma( \cL\dual \otimes \End \cL), \cQ ; H^1_2 \big)$, with the projection: 
\begin{gather*}
 \Pi^1_2: \Gamma( \cL\dual \otimes \End \cL) \onto \Gamma\big(\pi^* (B\dual \otimes \End B)\big)\cong \Gamma\big(\Lambda^\bullet A\dual \otimes \Hom(B\otimes B, B)\big) , \\
 \Pi^1_2(\Theta) = \pb \circ \Theta \circ (\tau  \otimes \tau),
\end{gather*}
where $\tau= \ib - \hpt \cQ \ib: \Gamma(\pi^* B) \to \Gamma(\cL)$, and $\cL\dual \otimes \End \cL$ is identified with $\Hom(\cL\otimes \cL, \cL)$. See Proposition~\ref{prop:HomTensorCont}.

Our main theorem is the following

\begin{theorem}\label{thm:Main:TwoAtiyahAreSame}
There exists an $\cL$-connection $\nabla^\cL$ on $\cL$ with the property: 
$$
\Pi^1_2(\cLAtiyah) = \pairAtiyah\nabla,
$$ 
where $\cLAtiyah$ is the Atiyah cocycle of the dg Lie algebroid $\cL$ associated with $\nabla^\cL$. 
In particular, the isomorphism 
$$
(\Pi^1_2)_\ast: H^1\big(\Gamma( \cL\dual \otimes \End \cL), \cQ \big) \xto\cong H_{\CE}^1(A,B\dual \otimes \End B)
$$
induced by $\Pi^1_2$ sends the Atiyah class of the dg Lie algebroid $\cL = \pi^! L$ to the Atiyah class of the Lie pair $(L,A)$. 
\end{theorem}

We will prove Theorem~\ref{thm:Main:TwoAtiyahAreSame} in Section~\ref{sec:AtiyahPullBackLieAbd}.

\subsection{A $(\pi^!L)$-connection on $\pi^!L$}

Let $\nabla: \Gamma(L) \times \Gamma(B) \to \Gamma(B)$ be an $L$-connection on $B$ extending the Bott connection. 
By choosing a splitting \eqref{eq_intro:splitting}, we further extend $\nabla$ to an $L$-connection $\tilde\nabla$ on $A[1]\oplus L$. 
The connection $\tilde \nabla: \Gamma(L) \times \Gamma(A[1] \oplus L) \to \Gamma(A[1] \oplus L)$ chosen in this way has the property 
$$
\tilde\nabla_{l} \big(\ib(b)\big) = \ib(\nabla_l b)   , \qquad \forall \, l \in \Gamma(L), b \in \Gamma(B).
$$

By choosing a connection, we have an isomorphism $\Gamma(\pi^!L) \cong \Gamma(\pi^* A[1]) \oplus \Gamma(\pi^* L)$. See \eqref{eq:TrivializedPullbackAbd}. We identify the graded vector bundle $\pi^*A[1]$ with the vertical tangent bundle $T_{A[1]}^{\ver}$ which is a graded Lie subalgebroid of $T_{A[1]}$.  
Let $\nabla^{A[1]}$ be a  $T_{A[1]}^{\ver}$-connection  on $\pi^! L$, and let 
$$
\nabla^\cL: \Gamma(\pi^! L) \times \Gamma(\pi^! L) \to \Gamma(\pi^! L)
$$
be the map 
$$
\nabla^\cL_{(x,v)} \lambda = \nabla^{A[1]}_x \lambda + \nabla^L_v \lambda, 
$$
for $(x,v) \in \Gamma(\pi^*A[1])\oplus \Gamma(\pi^*L) \cong \Gamma(\pi^!L)$, $\lambda \in \Gamma(\pi^!L)$, where 
$$
\nabla^L: \Gamma(\pi^* L) \times \Gamma\big(\pi^* (A[1]\oplus  L)\big) \to \Gamma\big(\pi^* (A[1]\oplus  L)\big),
$$
$$
 \nabla^L_{f \otimes l} \big(g \otimes (a,l')\big) = \big(\psi\big(f\otimes \anchorL(l)\big)(g) \big) \otimes (a,l') + (f g) \otimes \big(\tilde\nabla_l (a,l')\big),
$$
for $f,g \in C^\infty(A[1])$, $l,l' \in \Gamma(L)$ and $a \in \Gamma(A)$. In the definition of $\nabla^L$, we use the isomorphism $\psi$ between $\Gamma(\pi^* T_M)$ and the space of horizontal vector fields on $A[1]$ described in Remark~\ref{rmk:TwoExpressionpi!L}, and use the identification $\Gamma\big(\pi^* (A[1]\oplus  L)\big) \cong C^\infty(A[1]) \otimes_{C^\infty(M)} \Gamma(A[1]\oplus  L)$. Also note that $\nabla^L$ is well-defined because 
$$
\nabla^L_{f \otimes l} \big((bg) \otimes (a,l')\big) =  \big(\anchorL(l) b \big) \cdot fg \otimes (a,l') + b \cdot \nabla^L_{f \otimes l} \big(g \otimes (a,l')\big) = \nabla^L_{f \otimes l} \big(g \otimes (ba,bl')\big) 
$$
for any $b \in C^\infty(M)$.

\begin{lemma}\label{lem:pi^!Lcntn}
The bilinear map $\nabla^\cL$ is an $\cL$-connection on $\cL$ with the following properties: 
\begin{gather*}
\nabla^\cL_{(0,l\circ \pi)}(a\circ \pi,l'\circ \pi) = \big(\tilde\nabla_l (a,l')\big) \circ \pi, \\
\nabla^\cL_{(a\circ \pi, 0)} (a'\circ \pi, 0) =0,
\end{gather*}
for $a,a' \in \Gamma(A[1])$ and $l,l' \in \Gamma(L)$, where the pairs are elements of $\Gamma(\pi^*A[1])\oplus \Gamma(\pi^*L) \cong \Gamma(\pi^!L)$.  
In particular, 
$$
\nabla^\cL_{(0,l\circ \pi)}(0,\ib(b)\circ \pi) = \ib(\nabla_l b),
$$
for $b \in \Gamma(B)$. 
\end{lemma}
\begin{proof}
Since $\psi:\Gamma(\pi^* T_M) \to  \Gamma(T^{\hor}_{A[1]})$ is $C^\infty(A[1])$-linear, we have 
\begin{gather*}
\nabla^L_{f v} \lambda = f \nabla^L_{v} \lambda, \\
\nabla^L_{v} (f \lambda) = \big(\psi(\id\otimes \anchorL)(v) \big) (f) \cdot \lambda + f \nabla^L_{v} \lambda,
\end{gather*}
for $f \in C^\infty(A[1])$, $v \in \Gamma(\pi^*L)\cong C^\infty(A[1])\otimes \Gamma(L)$ and $\lambda \in\Gamma(\pi^!L)$. 
Note that, by Remark~\ref{rmk:TwoExpressionpi!L}, we have 
$$
\anchor(0,v) = \psi(\id\otimes \anchorL)(v),
$$
where $\anchor: \Gamma(\pi^!L) \cong \Gamma(\pi^*(A[1])\oplus \pi^* L) \to \XX(A[1])$ is the anchor map. 
Thus, for $f \in C^\infty(A[1])$, $v \in \Gamma(\pi^*L)$, $\lambda \in\Gamma(\pi^!L)$ and $x \in \Gamma(\pi^* A[1]) \cong \Gamma(T^{\ver}_{A[1]})$, 
\begin{gather*}
\nabla^\cL_{f (x,v)} \lambda = f \nabla^\cL_{(x,v)} \lambda, \\
\nabla^\cL_{(0,v)} (f \lambda) = \big(\anchor(0,v) \big) (f) \cdot \lambda + f \nabla^\cL_{(0,v)} \lambda, \\
\nabla^\cL_{(x,0)} (f \lambda) = \big(\anchor(x,0) \big) (f) \cdot \lambda + (-1)^{|f|} f \nabla^\cL_{(x,0)} \lambda,
\end{gather*}
where the last equation is from the fact $\nabla^{A[1]}$ is a $T_{A[1]}^{\ver}$-connection on $\pi^! L$. Therefore, $\nabla^\cL$ is an $\cL$-connection on $\cL$. 

The first property of $\nabla^\cL$ follows directly from the definition. For the second property, note that both $(a\circ \pi, 0)$ and $(a'\circ \pi, 0)$ are of degree $-1$, and thus $|\nabla^\cL_{(a\circ \pi, 0)} (a'\circ \pi, 0)|=-2$. Nevertheless, the degree of each homogeneous element in $\Gamma(\pi^!L)$ is at least $-1$. Thus, $\nabla^\cL_{(a\circ \pi, 0)} (a'\circ \pi, 0) =0$. 
\end{proof}

In \cite[Section~3.6]{2022arXiv221016769S}, Sti\'{e}non, Vitagliano and Xu independently constructed an $\cL$-connection on $\cL$ for a different purpose. One also can use their connection for Theorem~\ref{thm:Main:TwoAtiyahAreSame}.

\subsection{The two Atiyah classes}\label{sec:AtiyahPullBackLieAbd}

By Theorem~\ref{prop:ContPullbackLieAbd} and Proposition~\ref{prop:HomTensorCont}, we have the contraction 
\begin{equation}\label{eq:AtiyahContraction}
\begin{tikzcd}
\Big(\Gamma\big(\pi^*(B\dual \otimes \End B)\big),d^{\Bott}\Big)
\arrow[r, shift left,hook,"\cT^1_2"] &  \Big(\Gamma( \cL\dual \otimes \End \cL), \cQ \Big)
\arrow[l, " \Pi^1_2", shift left,two heads] \arrow[loop, "H^1_2",out=7,in= -7,looseness = 3]
\end{tikzcd},
\end{equation}
where 
\begin{equation}
\begin{split}
H^1_2& :\Gamma( \cL\dual \otimes \End \cL) \to \Gamma( \cL\dual \otimes \End \cL), \\
H^1_2(\Theta)  & = \hpt \circ \Theta + (-1)^{|\Theta|} \sigma \circ \Theta \circ \big(\hpt \otimes \id + \sigma \otimes \hpt \big),
\end{split}
\end{equation}
and
$$
\sigma =   (\ib - \hpt \cQ \ib)\pb = \id- [\cQ, \hpt] :\Gamma( \cL) \to \Gamma( \cL).
$$
The small space $\Gamma\big(\pi^*(B\dual \otimes \End B)\big)$ of the contraction \eqref{eq:AtiyahContraction} is identified with $\Gamma(\Lambda^\bullet A\dual \otimes B\dual \otimes \End B)$, equipped with the Bott differential $d^{\Bott}$. The projection map is defined as 
$$
\Pi^1_2: \Gamma( \cL\dual \otimes \End \cL) \onto \Gamma\big(\pi^*(B\dual \otimes \End B)\big), \; \Theta \mapsto \pb \circ \Theta \circ (\tau  \otimes \tau).
$$
Here, $\tau = \ib - \hpt \cQ \ib: \Gamma(\pi^*B) \to \Gamma( \cL)$, as defined in Theorem~\ref{prop:ContPullbackLieAbd}. 

Let $\pairAtiyah\nabla  \in \Gamma(\Lambda^1 A\dual \otimes B\dual \otimes \End B) \subset \Gamma\big(\pi^*(B\dual \otimes \End B)\big)$ be the Atiyah cocycle of the Lie pair $(L,A)$ associated with an $L$-connection $\nabla$ extending the Bott connection, and let $ \cLAtiyah \in \Gamma(\cL\dual \otimes \End \cL)$ be the Atiyah cocycle associated with the $\cL$-connection $\nabla^\cL$ constructed in Lemma~\ref{lem:pi^!Lcntn}. 
Let $\cLAtiyahProj \in \Gamma\big(\pi^*(B\dual \otimes \End B)\big)$ be the image  
$$
\cLAtiyahProj = \Pi^1_2(\cLAtiyah) = \pb \circ \cLAtiyah \circ (\tau \otimes \tau)
$$
of $\cLAtiyah$ under the projection $\Pi^1_2$. We will prove that $\cLAtiyahProj =\pairAtiyah\nabla$.

\begin{proof}[Proof of Theorem~\ref{thm:Main:TwoAtiyahAreSame}]
Following the notations in Remark~\ref{rmk:LocFormula}, we have 
\begin{align*}
\hpt \cQ   \ib(\bar e_l) 
& = \hpt\Big( \frac{1}{2}  \sum_{i,j,k=1}^r \rho_l(c_{ij}^k) \eta^i \eta^j \peta k   + \sum_{k=1}^{r+r'} \sum_{i=1}^r \eta^i c_{il}^k e_k \Big) \\
& = -\sum_{k=1}^{r} \sum_{i=1}^r \eta^i c_{il}^k \peta k,
\end{align*}
where $\bar e_l = \pb(e_l)$, $l = r+1, \cdots, r+r'$. Thus,
$$
\tau(\bar e_l) = \sigma(e_l) = 
e_l + \sum_{k=1}^{r } \sum_{i=1}^r \eta^i c_{il}^k \peta k,   \qquad \forall\, l= r+1, \cdots, r+r'.
$$

Let $\christoffelLL_{ij}^k$, $\christoffelLA_{ij}^k$, $\christoffelAA_{ij}^k$, $\christoffelAL_{ij}^k$ be the Christoffel symbols:  
\begin{gather*}
\nabla^\cL_{e_i} e_j =  \sum_{k=1}^{r+r'} \christoffelLL_{ij}^k e_k, \qquad \nabla^\cL_{e_i} \peta j = \sum_{k=1}^{r } \christoffelLA_{ij}^k \peta k,  \\
\nabla^{\cL}_{\peta i} \peta j = \sum_{k=1}^{r } \christoffelAA_{ij}^k \peta k =0, \qquad \nabla^{\cL}_{\peta i} e_j = \sum_{k=1}^{r } \christoffelAL_{ij}^k \peta k.
\end{gather*}
Note that since $\nabla^\cL_{e_i} e_j = (\tilde\nabla_{e_i} e_j)\circ \pi \in \Gamma(\pi^*L)$, we do not have $\peta k$-terms in $\nabla^\cL_{e_i} e_j$. 
Furthermore, it follows from Lemma~\ref{lem:pi^!Lcntn} that 
\begin{align}
\christoffelAA_{ij}^k &= 0,   \qquad \forall \, i,j,k, \\
\christoffelLL_{ij}^k  & = 0,   \qquad \forall \, j \geq r+1,\, k \leq r, \label{eq:LcntnLem} \\
\christoffelLL_{ij}^k   &=c_{ij}^k,   \qquad \forall \, i \leq r, \, j,k \geq r+1.\label{eq:LcntnBott}
\end{align}
For $i,j = r+1, \cdots, r+r'$, we have   
\begin{align*}
\cLAtiyahProj(\bar e_i,\bar e_j) & = \pb \cLAtiyah(\sigma e_i, \sigma e_j) \\
& =\pb \cLAtiyah(e_i +  \sum_{s,t=1}^r \eta^s c_{si}^t \peta t, e_j +  \sum_{u,v=1}^r \eta^u c_{uj}^v \peta v)  \\
& = \sA + \sB +\sC + \sD, 
\end{align*}
where
\begin{align*}
\sA &= \pb \cLAtiyah(e_i,e_j), \\
\sB & = \pb \cLAtiyah(e_i,   \sum_{u,v=1}^r \eta^u c_{uj}^v \peta v), \\
\sC & = \pb \cLAtiyah(\sum_{s,t=1}^r \eta^s c_{si}^t \peta t, e_j), \\
\sD & = \pb \cLAtiyah(\sum_{s,t=1}^r \eta^s c_{si}^t \peta t, \sum_{u,v=1}^r \eta^u c_{uj}^v \peta v).
\end{align*}

Using \eqref{eq:cQ_L}, \eqref{eq:LcntnLem} and \eqref{eq:LcntnBott}, one can show that 
\begin{align*}
\sA & = \sum_{p=1}^r \sum_{k=r+1}^{r+r'}   \eta^p \rho_p (\christoffelLL_{ij}^k) \bar e_k +  \sum_{p=1}^r \sum_{q,k =r+1}^{r+r'}  \eta^p \christoffelLL_{ij}^q  c_{pq}^k \bar e_k     - \sum_{p,q =1}^{r }\sum_{ k =r+1}^{r+r'}  \eta^p c_{pi}^q  {c}_{qj}^k \bar e_k \\
& \qquad  - \sum_{p=1}^r \sum_{q,k =r+1}^{r+r'}\eta^p c_{pi}^q  \christoffelLL_{qj}^k \bar e_k  - \sum_{p=1}^r\sum_{k=r+1}^{r+r'}    \eta^p \rho_i(c_{pj}^k) \bar e_k - \sum_{p=1}^r \sum_{q=1}^{r+r'}  \sum_{ k =r+1}^{r+r'} \eta^p  c_{pj}^q  \christoffelLL_{iq}^k \bar e_k.
\end{align*}
By \eqref{eq:cQ_A[1]}, \eqref{eq:cQ_L} and \eqref{eq:LcntnBott}, it is straightforward to show that 
$$
\pb \cLAtiyah(e_i,\peta v)  =  -  \sum_{k=r+1}^{r+r'} \christoffelLL_{iv}^k \bar e_k , \qquad \text{and} \qquad 
\pb \cLAtiyah(\peta t,e_j)  =  0.
$$
Thus, 
\begin{align*}
\sB & =  \sum_{p,q=1}^r \sum_{k=r+1}^{r+r'}  \eta^p c_{pj}^q      \christoffelLL_{iq}^k \bar e_k, \\
\sC &= 0. 
\end{align*}
Since the degree of $\cLAtiyah(\peta t,\peta v )$ is  $-1$, we have 
$$
\cLAtiyah(\peta t,\peta v ) \in \Gamma(\pi^*A[1]) \subset \ker(\pb),
$$
and thus 
$$
\sD = 0.
$$

Therefore, 
\begin{align*}
\cLAtiyahProj(\bar e_i,\bar e_j) & =   \sA + \sB + \sC + \sD \\
&= \sum_{p=1}^r \sum_{k=r+1}^{r+r'}   \eta^p \rho_p (\christoffelLL_{ij}^k) \bar e_k +  \sum_{p=1}^r \sum_{q,k =r+1}^{r+r'}  \eta^p \christoffelLL_{ij}^q  c_{pq}^k \bar e_k     - \sum_{p,q =1}^{r }\sum_{ k =r+1}^{r+r'}  \eta^p c_{pi}^q  {c}_{qj}^k \bar e_k \\
& \qquad  - \sum_{p=1}^r \sum_{q,k =r+1}^{r+r'}\eta^p c_{pi}^q  \christoffelLL_{qj}^k \bar e_k  - \sum_{p=1}^r\sum_{k=r+1}^{r+r'}    \eta^p \rho_i(c_{pj}^k) \bar e_k - \sum_{p=1}^r   \sum_{q, k =r+1}^{r+r'} \eta^p  c_{pj}^q  \christoffelLL_{iq}^k \bar e_k.
\end{align*}

For $p=1, \cdots, r$ and $i,j = r+1, \cdots, r+r'$, we have
\begin{align*}
\pairAtiyah\nabla(e_p,\bar e_i) \bar e_j & = \nabla_{e_p} \nabla_{e_i} \bar e_j - \nabla_{e_i} \nabla_{e_p} \bar e_j - \nabla_{[e_p,e_i]} \bar e_j \\
& = \sum_{k=r+1}^{r+r'} \nabla_{e_p}  \big(\christoffelLL_{ij}^k \bar e_k \big) - \sum_{k=r+1}^{r+r'}  \nabla_{e_i} \big( c_{pj}^k \bar e_k \big)- \sum_{l= 1}^{r+r'}  c_{pi}^l \nabla_{e_l} \bar e_j \\
& = \sum_{k=r+1}^{r+r'} \Big(  \rho_p \big(\christoffelLL_{ij}^k \big) \bar e_k +\sum_{l=r+1}^{r+r'} \christoffelLL_{ij}^k c_{pk}^l \bar e_l   \Big) - \sum_{k=r+1}^{r+r'} \Big( \rho_i(c_{pj}^k) \bar e_k+  \sum_{l=r+1}^{r+r'}  c_{pj}^k \christoffelLL_{ik}^l \bar e_l  \Big)\\
& \qquad - \sum_{k=r+1}^{r+r'} \sum_{l= 1}^{r }  c_{pi}^l c_{lj}^k \bar e_k - \sum_{k=r+1}^{r+r'} \sum_{l= r+1}^{r+r'}  c_{pi}^l \christoffelLL_{lj}^k \bar e_k .
\end{align*}
Equivalently,  
\begin{align*}
\pairAtiyah\nabla(\bar e_i, \bar e_j) & = \sum_{p=1}^r \sum_{k=r+1}^{r+r'} \eta^p \rho_p \big(\christoffelLL_{ij}^k \big) \bar e_k +\sum_{p=1}^r   \sum_{q,k=r+1}^{r+r'} \eta^p \christoffelLL_{ij}^q  c_{pq}^k \bar e_k     - \sum_{p=1}^r \sum_{k=r+1}^{r+r'}   \eta^p \rho_i(c_{pj}^k) \bar e_k  \\ 
& \qquad -\sum_{p=1}^r  \sum_{q,k=r+1}^{r+r'} \eta^p c_{pj}^q \christoffelLL_{iq}^k \bar e_k    - \sum_{p,q= 1}^{r } \sum_{k=r+1}^{r+r'}  \eta^p  c_{pi}^q c_{qj}^k \bar e_k - \sum_{p=1}^{r  } \sum_{q,k= r+1}^{r+r'} \eta^p c_{pi}^q \christoffelLL_{qj}^k \bar e_k . 
\end{align*}
By comparing the formulas of $\cLAtiyahProj(\bar e_i,\bar e_j)$ and $\pairAtiyah\nabla(\bar e_i, \bar e_j)$, we conclude that 
$$
\Pi^1_2(\cLAtiyah) = \cLAtiyahProj = \pairAtiyah\nabla.
$$ 
The proof of Theorem~\ref{thm:Main:TwoAtiyahAreSame} is complete. 
\end{proof}

\subsection{The two Todd classes}

By Corollary~\ref{cor:ToddContraction}, the contraction \eqref{eq:piL_contraction} generates the contraction data
$$
\begin{tikzcd}
\big(\Gamma\big(\pi^* (\Lambda B\dual \otimes \End B)\big),d^{\Bott} \big)
\arrow[r, " \widehat{\cT} ", shift left,hook] &  \big(\Gamma(\Lambda \cL\dual \otimes \End \cL), \cQ \big)
\arrow[l, " \widehat{\Pi}", shift left,two heads] \arrow[loop, "\widehat{H}",out=8,in= -8,looseness = 3]
\end{tikzcd}
$$
whose inclusion map $\widehat{\cT}$ is an algebra morphism.  
Furthermore, we also have the contraction data  
\begin{equation}\label{eq:WedgeContraction}
\begin{tikzcd}
\big(\Gamma\big(\pi^* (\Lambda B\dual)\big),d^{\Bott} \big)
\arrow[r, " \cT_\Lambda", shift left,hook] &  \big(\Gamma(\Lambda \cL\dual ), \cQ \big)
\arrow[l, " \Pi_\Lambda", shift left,two heads] \arrow[loop, "H_\Lambda",out=8,in= -8,looseness = 3]
\end{tikzcd}
\end{equation}
by Lemma~\ref{lem:HomContraction} and Proposition~\ref{prop:ExteriorAlgContraction}. 

\begin{lemma}\label{lem:Trace}
The inclusion maps $\widehat\cT$ and $\cT_\Lambda$ commute with the (super)traces: 
$$
\begin{tikzcd}
\Gamma\big(\pi^* (\Lambda B\dual \otimes \End B)\big) \ar[r,"\widehat\cT"]  \ar[d,"\tr"'] & \Gamma(\Lambda\cL\dual \otimes \End \cL) \ar[d,"\str"] \\
\Gamma\big(\pi^* (\Lambda B\dual)\big) \ar[r,"\cT_\Lambda"] &  \Gamma(\Lambda\cL\dual) 
\end{tikzcd}
$$ 
\end{lemma}
\begin{proof}
Recall that for $\omega   \in \Gamma\big(\pi^* (\Lambda B\dual)\big)$ and $\Phi \in \Gamma\big(\pi^* (\End B)\big)$, we have $\widehat\cT(\omega \otimes \Phi) = \cT_\Lambda(\omega) \otimes (\tau \Phi \pb)$, where $\tau   = \ib - \hpt \partial \, \ib $. Since $\im(\hpt \partial) \subset \Gamma(\pi^* A[1]) \subset \Gamma(\cL) \cong \Gamma(\pi^* A[1]) \oplus \Gamma(\pi^* L)$, the matrix representation of  $\tau \Phi \pb$ is of the form
$$
\begin{pmatrix}
\ib \Phi \pb & \; 0 \quad \\
- \hpt \partial \, \ib \Phi \pb & \; 0 \quad
\end{pmatrix},
$$
where the first row/column represents the even component $\Gamma(\pi^* L)$, and the second row/column represents the odd component $\Gamma(\pi^* A[1])$.
Thus, 
$$
\str(\tau \Phi \pb) = \str(\ib \Phi \pb) = \tr(\Phi).
$$
As a result, we have
$$
\str \widehat\cT(\omega \otimes \Phi) = \cT_\Lambda(\omega)  \str(\tau \Phi \pb)= \cT_\Lambda(\omega)  \tr(\Phi) = \cT_\Lambda(\tr(\omega \otimes \Phi)).
$$
This completes the proof.
\end{proof}

\begin{theorem}\label{thm:Todd}
The projection map $\Pi_\Lambda:\Gamma(\Lambda \cL\dual) \to \Gamma(\pi^* \Lambda B\dual)$ induces an isomorphism 
$$
(\Pi_\Lambda)_* :H^\bullet\big(\Gamma(\Lambda^k \cL\dual ), \cQ \big) \xto{\cong}  H_{\CE}^\bullet\big(A,\Lambda^k B\dual \big)
$$
for each $k$, and 
\begin{equation}\label{eq:2Todd}
(\Pi_\Lambda)_*(\Todd_\cL) = \Todd_{L/A}.
\end{equation}
\end{theorem}
\begin{proof}
Note that all the operators in the contraction \eqref{eq:WedgeContraction} respect to $\Lambda^k$. Thus, one can decompose \eqref{eq:WedgeContraction} to contractions for $\Lambda^k$, and the first assertion follows. 

Equation \eqref{eq:2Todd} is equivalent to 
$$
(\cT_\Lambda)_*(\Todd_{L/A}) =\Todd_\cL.
$$
Since the Todd classes can be expressed in terms of scalar Atiyah classes, it suffices to show that 
$$
(\cT_\Lambda)_*(\tr \alpha_{L/A}^k) =\str \alpha_{\cL}^k,
$$
for each $k$. Since $\widehat{\cT}_*$ is an algebra isomorphism, it follows from Lemma~\ref{lem:Trace} and Theorem~\ref{thm:Main:TwoAtiyahAreSame} that 
$$
(\cT_\Lambda)_*(\tr \alpha_{L/A}^k) = \str \widehat\cT_*(\alpha_{L/A}^k)= \str\big( \big((\cT^1_2)_*(\alpha_{L/A})\big)^k\big) =\str \alpha_{\cL}^k,
$$
where $(\cT^1_2)_*$ is induced by the contraction \eqref{eq:AtiyahContraction}.
\end{proof}

\subsection{Applications}\label{sec:Application}

\subsubsection{Integrable distributions}

Let $L=T_M$ be the tangent bundle of a manifold $M$. Let $A=F \subset T_M$ be a Lie subalgebroid whose sections form an integrable distribution. The pullback Lie algebroid 
$$
\pi^! T_M = T_{F[1]}\times_{T_M} T_M = T_{F[1]}
$$
can be identified with the Lie algebroid $T_{F[1]}$. Furthermore, the dg structure on $\Gamma( T_{F[1]})$ is given by $[s_{i_F},\argument]$, where $s_{i_F} = d_F \in \Gamma(T_{F[1]})$ is the Chevalley--Eilenberg differential. 
See Proposition~\ref{prop:DGLieAbdFromLiePair} and Lemma~\ref{lem:SecFromLieAbdMor}. Thus, by Theorem~\ref{prop:ContPullbackLieAbd}, we have the contraction data 
\begin{equation}\label{eq:CXX_contraction}
\begin{tikzcd}
\big(\Gamma(\Lambda^\bullet F\dual \otimes B),d^{\Bott}\big)
\arrow[r, " \tau ", shift left,hook] &  \big(\Gamma(T_{F[1]}), L_{d_F} \big)
\arrow[l, " \pb", shift left,two heads] \arrow[loop, "\hpt",out=12,in= -12,looseness = 3]
\end{tikzcd},
\end{equation}
where $B = T_M/F$. 
This contraction \eqref{eq:CXX_contraction} coincides with Chen--Xiang--Xu's contraction in \cite[Lemma~2.2]{MR3877426}. 

As a consequence of Theorem~\ref{thm:Main:TwoAtiyahAreSame}, we have the following 

\begin{corollary}
The contraction \eqref{eq:CXX_contraction} induces an isomorphism 
$$
H^1(\Gamma(T_{F[1]}\dual\otimes \End T_{F[1]}), L_{d_F}) \xto\cong H_{\CE}^1(F,B\dual\otimes \End B)
$$
which sends the Atiyah class of the dg manifold $F[1]$ to the Atiyah class of the Lie pair $(F,T_M)$.
\end{corollary}

Similarly, by Theorem~\ref{thm:Todd}, we have 

\begin{corollary}
The contraction \eqref{eq:CXX_contraction} induces an isomorphism 
$$
\prod_{k=0}^\infty H^k(\Gamma(\Lambda^k T_{F[1]}\dual), L_{d_F}) \xto\cong  \prod_{k=0}^\infty H_{\CE}^k(F,\Lambda^k B\dual)
$$
which sends the Todd class of the dg manifold $F[1]$ to the Todd class of the Lie pair $(F,T_M)$.
\end{corollary}

We recover Chen--Xiang--Xu's theorems in \cite{MR3877426}. 
In particular, the Atiyah class and the Todd class of a complex manifold $X$ can be identified with the Atiyah class and the Todd class of the dg manifold $T_X^{0,1}[1]$, respectively. See \cite[Theorem~C]{MR3877426}.

\subsubsection{$\frakg$-manifolds}

Let $\frakg$ be a finite-dimensional Lie algebra. 
A {\bf $\frakg$-manifold} is a smooth manifold $M$ together with a Lie algebra action, i.e.\ a morphism of Lie algebras $\frakg \ni a \mapsto \hat{a} \in \XX(M)$. 
Given a $\frakg$-manifold $M$, the action Lie algebroid $\frakg \ltimes M$ and the tangent bundle $T_M$ naturally form a matched pair of Lie algebroids \cite{MR1460632}. Thus, we have a Lie pair $(L,A)$, where $L = (\frakg \ltimes M)\bowtie T_M$ and $A = \frakg \ltimes M$. 
See \cite{MR2534186} for its relation with BRST complexes.

More explicitly, the vector bundle $L$ is isomorphic to $(\frakg \times M) \oplus T_M$ as vector bundles over $M$, and the anchor $\anchorL: \Gamma(L) \to \XX(M)$ is given by the formula $\anchorL(a + X) = \hat{a} +X$. The bracket $[\argument,\argument]:\Gamma(L) \times \Gamma(L) \to \Gamma(L)$ is determined by 
$$
[a,b] = [a,b]_\frakg, \qquad [X,Y] = [X,Y]_{\XX(M)}, \qquad [a,X] = [\hat{a},X].
$$
Here $a,b \in \frakg$ are identified with the corresponding constant functions in $ C^\infty(M,\frakg)\cong \Gamma(A)$, $X$ and $Y$ are vector fields on $M$, $[\argument,\argument]_\frakg$ is the Lie bracket of $\frakg$, and $[\argument,\argument]_{\XX(M)}$ is the Lie bracket of vector fields.

Let $B$ be the quotient vector bundle $B=L/A \cong T_M$. In this case, the graded vector bundle $\cL = \pi^!L$ admits a natural Whitney sum decomposition over $A[1] \cong \frakg[1] \times M$:
$$
\cL \cong \pi^\ast A[1] \oplus \pi^\ast A \oplus \pi^\ast B,
$$
where 
$$
\pi^\ast A[1] \cong  ( \frakg[1] \times M) \times \frakg[1], \quad \pi^\ast A \cong ( \frakg[1] \times M) \times \frakg \quad \text{and} \quad \pi^\ast B \cong \frakg[1] \times T_M.
$$
Consequently, its space of sections admits the decomposition
$$
\Gamma(\cL) \cong ( \Lambda^\bullet\frakg\dual \otimes C^\infty(M)   \otimes \frakg[1]) \oplus (\Lambda^\bullet\frakg\dual \otimes C^\infty(M)   \otimes \frakg) \oplus (\Lambda^\bullet \frakg\dual \otimes \XX(M)). 
$$

Now we describe the contraction \eqref{eq:piL_contraction} in this situation. By \eqref{eq:cQ_A[1]} and \eqref{eq:cQ_L}, a direct computation shows that the differential $\cQ$ acts on $\Gamma(\cL)$ as follows:
$$
\begin{tikzcd}
\Gamma(\cL) \cong (\Lambda^\bullet\frakg\dual \otimes C^\infty(M)   \otimes \frakg[1]) \oplus (\Lambda^\bullet\frakg\dual \otimes C^\infty(M)   \otimes \frakg) \oplus (\Lambda^\bullet \frakg\dual \otimes \XX(M)) \ar[d,xshift=-5.8cm,"\cQ"'] \ar[d,start anchor={[xshift=-2.7cm]}, end anchor={[xshift=0.8cm]}] \ar[d,xshift=-3.2cm] \ar[d,xshift=1.3cm] \ar[d,xshift=4.9cm] \\
\Gamma(\cL) \cong (\Lambda^\bullet\frakg\dual \otimes C^\infty(M)   \otimes \frakg[1]) \oplus (\Lambda^\bullet\frakg\dual \otimes C^\infty(M)   \otimes \frakg) \oplus (\Lambda^\bullet \frakg\dual \otimes \XX(M))
\end{tikzcd}
$$
The homotopy operator $\hpt: \Gamma(\cL) \to \Gamma(\cL)$ is linear over $C^\infty(A[1]) \cong \Lambda^\bullet\frakg\dual \otimes C^\infty(M)$ and is determined by its values on the three components $\Gamma(\pi^\ast A[1])$, $\Gamma(\pi^\ast A)$ and $\Gamma(\pi^\ast B)$. The values of $\hpt$ on the components $\Gamma(\pi^\ast A[1])$ and $\Gamma(\pi^\ast B)$ vanish. On the component $\Gamma(\pi^\ast A)$, it is given by extending the degree-shifting map $\nshift:\frakg \to \frakg[1]$ via $C^\infty(A[1])$-linearity.  

For the small complex $(\Gamma(\pi^\ast B), d^{\Bott})$, it is clear that $\Gamma(A) \cong C^\infty(M,\frakg)$ and $\Gamma(B) \cong \XX(M)$. The Bott connection $\nabla^{\Bott}:C^\infty(M,\frakg) \times \XX(M) \to \XX(M)$ is determined by the formula 
$$
\nabla^{\Bott}_{a} X = [\hat{a},X], \qquad \forall\, a \in \frakg, \, X \in \XX(M),
$$ 
where an element $a \in \frakg$ is identified with the constant function with value $a$. 
The complex $(\Gamma(\pi^* B),d^{\Bott})$ coincides with the Chevalley--Eilenberg complex $(\Lambda^\bullet \frakg\dual \otimes \XX(M),d_{\CE})$ of the $\frakg$-action $\frakg \to \End\XX(M), \, a \mapsto L_{\hat{a}}$. 

According to Theorem~\ref{prop:ContPullbackLieAbd}, the projection map $\pb:\Gamma(\cL) \to \Gamma(\pi^\ast B)$ of the contraction \eqref{eq:piL_contraction} is the canonical projection onto $\Gamma(\pi^\ast B)$. 
Since the subset $ 0 \oplus 0 \oplus \Gamma(\pi^\ast B) \subset \Gamma(\cL)$ is stable under $\cQ$, it follows from \eqref{eq:tauMain} that the map $\tau$ coincides with the canonical inclusion in the case of $\frakg$-manifolds. 
As a consequence, the projection maps in the contractions \eqref{eq:AtiyahContraction} and \eqref{eq:WedgeContraction} are the canonical projections. 

By Theorem~\ref{thm:Main:TwoAtiyahAreSame} and Theorem~\ref{thm:Todd}, we have the following 

\begin{corollary}
The canonical projections 
$$\Pi^1_2: \big(\Gamma(\cL\dual \otimes \End \cL), \cQ\big)  \to \big(\Lambda^\bullet\frakg\dual \otimes \Gamma(T_M\dual \otimes \End T_M), d_{\CE}\big)
$$
and
$$
\Pi_\Lambda : \big( \Gamma(\Lambda^k \cL\dual), \cQ \big)  \to  \big( \Lambda^\bullet\frakg\dual \otimes \Gamma(\Lambda^k T_M\dual), d_{\CE}\big)
$$
are quasi-isomorphisms whose induced maps on cohomologies send the Atiyah class and the Todd class of the dg Lie algebroid $\cL=\pi^! L$
to the Atiyah class and the Todd class of the Lie pair $(L,A) = ((\frakg \ltimes M)\bowtie T_M,\frakg \ltimes M)$, respectively.
\end{corollary}

\appendix

\section{Contractions over a dg ring}\label{sec:HPL}

A contraction is an algebraic analogue of a deformation retract. The key proposition for us is the homological perturbation lemma which is an algebraic tool for perturbing a contraction to another contraction. See, for example, \cite{MR56295,MR2760671,2004math......3266C}. 
In this appendix, we summarize the necessary facts about contractions and the homological perturbation lemma.

We formulate contractions by characterizing the homotopy operator. 
In this formulation, one needs only a complex and a homotopy operator satisfying certain conditions, while in the classical formulation \cite{MR56295}, one needs a small complex, a projection map and an inclusion map in addition. 
One can find an $L_\infty$ version of our formulation in \cite[Appendix~B]{2020arXiv200601376B}. 
In Section~\ref{sec:CompareHLP}, we describe how one can generate the additional data in the usual definition of contraction data from the homotopy operator.

\subsection{Homological perturbation lemma}

Let $R =(R,d_R)$ be a commutative dg ring.
 
\begin{definition}\label{def:ContractionCptDef}
A \textbf{contraction} over $R$ is a triple $(W,\delta; h)$, where $(W,\delta)$ is a dg module over $R$, and $h$ is an  $R$-linear operator $h:W^p \to W^{p-1}$ of degree $-1$ such that 
$$
h^2 =0 \qquad \text{and} \qquad h \delta h =h.
$$
The operator $h$ is referred to as the \textbf{homotopy operator} of the contraction $(W,\delta; h)$. 
A \textbf{perturbation} $\partial$ of $(W,\delta)$ over $R$ is an $R$-linear  operator $\partial: W^p \to W^{p+1}$ of degree $+1$ such that 
$$
(\delta + \partial)^2 =0.
$$
\end{definition}

Note that if $\partial$ is a perturbation
of a dg module $(W,\delta)$ over $R$, then $(W,\delta + \partial)$
is also a dg module over $R$.

\begin{definition}
We say a perturbation $\partial$ is a \textbf{small perturbation} of a contraction $(W,\delta; h)$ if there exists a  descending exhaustive (i.e. $\cup_q F^q(W) = W$) complete (i.e. $W = \varprojlim_q W/F^qW$) filtration 
$$
F = \cdots \supset F^q W \supset F^{q+1} W \supset \cdots
$$ 
of the space $W$ (\emph{not} necessarily compatible with $\delta$) such that 
$$
(\partial h)(F^q W) \subset F^{q+1}W, \qquad \forall q.
$$
\end{definition}

The following theorem is well-known.

\begin{theorem}[Homological perturbation lemma]\label{thm:HPLCptForm}
Let $\partial$ be a small perturbation of a contraction $(W,\delta;h)$ over $R$.
Then (i) the operators $\id+h\partial$ and $\id+\partial h$ are invertible
and their inverses are the convergent series
\[ (\id+h\partial)^{-1}=\sum_{k=0}^\infty (-h\partial)^k
\qquad\text{and}\qquad
(\id+\partial h)^{-1}=\sum_{k=0}^\infty (-\partial h)^k ;\]
(ii) the operator
\begin{equation}\label{eq:PertHpt}
h_\partial:=(\id+h\partial)^{-1}h=h(\id+\partial h)^{-1}
\end{equation}
is well-defined; and (iii) the triple $(W,\delta +\partial; h_\partial)$
forms a contraction over $R$.
\end{theorem}

The contraction $(W,\delta +\partial; h_\partial)$ is referred to as the \textbf{perturbed contraction}.

\begin{proof}
The first two assertions are immediate.
It follows from the equation \eqref{eq:PertHpt} and $h^2=0$ that
\[ h_\partial h_\partial=(\id+h\partial)^{-1}hh(\id+\partial h)^{-1}=0 .\]
From $h\delta h=h$, we obtain
\[ h(\delta+\partial)h=h+h\partial h=h(\id+\partial h) .\]
It follows from the equation \eqref{eq:PertHpt} that
\[ h_\partial(\delta+\partial)h_\partial
=(\id+h\partial)^{-1}h(\delta+\partial)h(\id+\partial h)^{-1}
% =(\id+h\partial)^{-1}h(\id+\partial h)(\id+\partial h)^{-1}
=(\id+h\partial)^{-1}h=h_\partial .\qedhere \]
\end{proof}

\subsection{Classical formulation of contractions}\label{sec:CompareHLP}

Let $(V,d)$ and $(W,\delta)$ be two dg modules over a commutative dg ring $R$. 
A \textbf{contraction data} is the data   
$$
\begin{tikzcd}
(V,d) \arrow[r, " \tau", shift left,hook] &  (W ,\delta)
\arrow[l, " \sigma", shift left,two heads] \arrow[loop, "h",out=12,in= -12,looseness = 3]
\end{tikzcd}
$$ 
where $\tau:V\to W$ is an injective $R$-linear cochain map,
$\sigma:W\to V$ is a surjective $R$-linear cochain map, 
$h: W \to W$ is an $R$-linear map of degree $-1$, and 
\begin{gather*}
\sigma\tau=\id_V, \qquad \id_W-\tau\sigma=h\delta+\delta h, \\
\sigma h=0, \qquad h\tau=0, \qquad h h=0 .
\end{gather*} 
The space $V$ is referred to as the \textbf{small space} of the contraction data.
The maps $\tau$ and $\sigma$ will be referred to respectively as the \textbf{injection} (or the \textbf{inclusion map})
and the \textbf{surjection} (or the \textbf{projection map}) of the contraction. 
We refer the reader to \cite{MR2760671} for the basic properties of contraction data.

Let $(W,\delta; h)$ be a contraction in the sense of Definition~\ref{def:ContractionCptDef}.
Since $h\delta h=h$, $hh=0$, and $\delta\delta=0$,
the operators $\delta h$, $h \delta$ and $[\delta,h]$ are projection operators.
Consider the subspace
\[ V:=\im(\id - [\delta,h])=\ker [\delta, h]= \ker(\delta h) \cap \ker(h \delta) \]
of $W$, let $\tau:V\into W$ denote the inclusion of $V$ into $W$,
and let $\sigma:W\onto V$ the surjection induced by the projection
operator $\varpi:=\id-[\delta,h]$.
Note that, since $[\delta,h]$ is $R$-linear, $V$ is an $R$-module.
Furthermore, since $\delta^2=0$, we have 
\[ \varpi\delta=(\id-h\delta-\delta h)\delta
=\delta-\delta h\delta=\delta(\id-h\delta-\delta h)=\delta\varpi ,\]
which shows that $V=\im(\varpi)$ is a subcomplex of $(W,\delta)$.

The next lemma follows from a direct computation. 

\begin{lemma}
The data 
$$
\begin{tikzcd}
(V,\delta|_V)
\arrow[r, " \tau", shift left,hook] &  (W ,\delta)
\arrow[l, " \sigma", shift left,two heads] \arrow[loop, "h",out=12,in= -12,looseness = 3]
\end{tikzcd}
$$
induced by $(W,\delta;h)$ forms a contraction data.
\end{lemma}

Let $F$ be an exhaustive complete filtration on $W$,
and $\partial$ be a perturbation of $(W,\delta)$
which satisfies the assumptions of Theorem~\ref{thm:HPLCptForm}.
Consequently, the operators $\id +\partial h$ and $\id +h\partial$ are invertible
and we have a perturbed contraction $(W,\delta+\partial; h_\partial)$
with induced subcomplex
$$
V_\partial := \im(\id_W - [\delta+\partial , h_\partial]) 
=\ker([\delta+\partial , h_\partial]) 
= \ker\big((\delta+\partial)h_\partial\big) \cap \ker\big(h_\partial(\delta+\partial)\big)
,$$
the image of the projection operator $\varpi_\partial=\id-[\delta+\partial,h_\partial]$.

Hence, we obtain the contraction data
\begin{equation}\label{uranus} \begin{tikzcd}
(V_\partial, (\delta+\partial)|_{V_\partial})
\arrow[r, shift left,hook, "\psi"] &  (W ,\delta+\partial)
\arrow[l, shift left,two heads, "\varphi"]
\arrow[loop, "h_\partial",out=12,in= -12,looseness = 3]
\end{tikzcd} \end{equation}
where $\psi$ denotes the inclusion of $V_\partial=\im(\varpi_\partial)$ into $W$
and $\varphi:W\onto V_\partial$ is the surjection induced by the projection operator
$\varpi_\partial$.

We note that
\begin{equation}\label{sun}
(\id+\partial h)^{-1}=\sum_{k=0}^\infty (-\partial h)^k=\id-\partial h_\partial
\end{equation}
and
\begin{equation}\label{moon}
(\id+h \partial)^{-1}=\sum_{k=0}^\infty (-h\partial)^k=\id- h_\partial \partial
.\end{equation}

\begin{lemma}\label{neptune}
The following diagram is commutative:
\[ \begin{tikzcd}[column sep=huge,row sep=huge]
W \arrow[dd, bend right=75, "(\id+h\partial)^{-1}"'] && 
W \arrow[ll, "\varpi=\id-h\delta-\delta h"'] \arrow[ld, "\id-\delta h"] \\
& W \arrow[lu, "\id-h\delta"] \arrow[ld, "\id-h_\partial(\delta+\partial)"'] & \\
W && W \arrow[uu, bend right=75, "(\id+\partial h)^{-1}"']
\arrow[lu, "\id-(\delta+\partial)h_\partial"']
\arrow[ll, "\varpi_\partial=\id-h_\partial(\delta+\partial)-(\delta+\partial)h_\partial"]
\end{tikzcd} \]
\end{lemma}

\begin{proof}
It follows immediately from $\delta^2=0$
that $(\id-h\delta)(\id-\delta h)=\varpi$.

Likewise, it follows immediately from $(\delta+\partial)^2=0$
that $\big(\id-h_\partial(\delta+\partial)\big)
\big(\id-(\delta+\partial)h_\partial\big)=\varpi_\partial$.

It follows from the equations \eqref{sun} and \eqref{eq:PertHpt} that
\begin{equation}\label{eq:1-dh}
\id-(\delta+\partial)h_\partial=(\id-\partial h_\partial)-\delta h_\partial
=(\id+\partial h)^{-1}-\delta h(\id+\partial h)^{-1}=(\id-\delta h)(\id+\partial h)^{-1}
.\end{equation}

Likewise, it follows from the equations \eqref{moon} and \eqref{eq:PertHpt} that
\begin{equation}\label{eq:1-hd}
\id-h_\partial(\delta+\partial)=(\id-h_\partial \partial)-h_\partial\delta
=(\id+h\partial)^{-1}-(\id+h\partial)^{-1}h\delta=(\id+h\partial)^{-1}(\id-h\delta)
.\end{equation}

The proof is complete.
\end{proof}

In the diagram of Lemma~\ref{neptune}, all straight arrows are projection operators,
while the two bended arrows are automorphisms of the graded $R$-module $W$.
Since \[ (\id+h\partial)^{-1}\circ\varpi\circ(\id+\partial h)^{-1}=\varpi_\partial ,\]
the automorphism $(\id+h\partial)\inv$ of the $R$-module $W$
identifies the submodules $V=\im(\varpi)$ and $V_\partial=\im(\varpi_\partial)$.

Hence, we obtain the commutative diagram
% \[ \begin{tikzcd}[column sep=huge,row sep=huge]
% W \arrow[ddd, "(\id+h\partial)^{-1}"'] && 
% W \arrow[ll, bend right,
% "\varpi=\id-h\delta-\delta h"'] \arrow[ld, two heads, "\sigma"] \\
% & W \arrow[lu, hook, "\tau"] \arrow[d] & \\
% & W \arrow[ld, hook] & \\
% W && W \arrow[uuu, "(\id+\partial h)^{-1}"'] \arrow[lu, two heads]
% \arrow[ll, bend left,
% "\varpi_\partial=\id-h_\partial(\delta+\partial)-(\delta+\partial)h_\partial"]
% \end{tikzcd} \]
\begin{equation}\label{pluto}
\begin{tikzcd}[column sep=huge,row sep=huge]
W \arrow[d, "(\id+h\partial)^{-1}"', "\cong"] &
V \arrow[l, hook, "\tau"'] \arrow[d, shift right, "(\id+h\partial)^{-1}"'] & 
W \arrow[ll, bend right,
"\varpi=\id-h\delta-\delta h"'] \arrow[l, two heads, "\sigma"'] \\
W & V_\partial \arrow[l, hook, "\psi"] \arrow[u, shift right, "\id+h\partial"'] &
W \arrow[u, "(\id+\partial h)^{-1}"', "\cong"] \arrow[l, two heads, "\varphi"]
\arrow[ll, bend left,
"\varpi_\partial=\id-h_\partial(\delta+\partial)-(\delta+\partial)h_\partial"]
\end{tikzcd}
\end{equation}

\begin{definition}\label{def:PertOp}
The composition $V \xto{(\id+h\partial)^{-1}} V_\partial \xto{\psi} W$ will be referred to as
the \textbf{perturbed injection} and denoted by $\tau_\partial$.
The composition $W \xto{\varphi} V_\partial \xto{\id+h\partial} V$ 
will be referred to as the \textbf{perturbed surjection} and denoted by $\sigma_\partial$.
Since $V_\partial$ is stable under the differential $\delta+\partial$,
the composition $(\id+h\partial)\circ(\delta+\partial)\circ(\id+h\partial)^{-1}$
stabilizes $V$ and determines a differential on $V$, which we will refer to as
the \textbf{perturbed differential} and denote by $\delta_\partial$.
\end{definition}

From \eqref{uranus} and \eqref{pluto},
we obtain the contraction data
\begin{equation}\begin{tikzcd}
(V, \delta_\partial)
\arrow[r, shift left, hook, "\tau_\partial"] &  (W ,\delta+\partial)
\arrow[l, shift left, two heads, "\sigma_\partial"]
\arrow[loop, "h_\partial",out=12,in= -12,looseness = 3]
\end{tikzcd} \end{equation}

\begin{corollary}\label{cor:FormulaPertOp}
Under the assumptions of Theorem~\ref{thm:HPLCptForm}, we have
\begin{gather*}
\tau_\partial = (\id + h\partial)\inv \circ \tau 
= \sum_{k=0}^\infty (-h \partial)^k \tau \\
\sigma_\partial = \sigma \circ (\id + \partial h)\inv 
= \sum_{k=0}^\infty \sigma (-\partial h)^k \\
\delta_\partial - \delta|_V 
= \sigma \circ\partial \circ (\id + h \partial)\inv \tau 
= \sum_{k=0}^\infty \sigma \partial(-h\partial)^k \tau.
\end{gather*}
\end{corollary}

\begin{proof}
The equations for $\tau_\partial$ and $\sigma_\partial$ follow immediately
from the commutativity of the diagram~\eqref{pluto}.

According to Definition~\ref{def:PertOp}, we have
\[ \tau\delta_\partial =(\id+h\partial)\circ(\delta+\partial)\circ(\id+h\partial)^{-1}\circ \tau
 .\]
Since $\sigma\tau=\id_V$, it follows that
\[ \delta_\partial=\sigma\circ(\id+h\partial)\circ
(\delta+\partial)\circ(\id+h\partial)^{-1}\circ\tau .\]
Since $\sigma h=0$, the equation simplifies to
\[ \delta_\partial=\sigma\circ
(\delta+\partial)\circ(\id+h\partial)^{-1}\circ\tau \]
or, equivalently,
\[ \delta_\partial=\sigma\circ
(\delta+\partial)\circ(\id-h_\partial \partial)\circ\tau .\]
Since $h_\partial=h (\id+\partial h)^{-1}$ and $\sigma h=0$, we obtain
\[ \sigma\delta(\id-h_\partial \partial)\tau
=\delta|_V\sigma(\id-h_\partial \partial)\tau
=\delta|_V\sigma\tau=\delta|_V .\]
Therefore, we conclude that
\[ \delta_\partial-\delta|_V=\sigma\circ
\partial\circ(\id-h_\partial \partial)\circ\tau = \sigma\circ
\partial\circ(\id + h \partial)^{-1}\circ\tau  .\qedhere \]
\end{proof}

%\begin{comment}
%\begin{proof}
%The formula of $\tau_\partial$ is clear. For the formula of $\sigma_\partial$, applying \eqref{eq:1-hd} and $(\delta + \partial)^2=0$, we have  
%\begin{align*}
%\sigma_\partial & =(\id + h\partial)  \Big((\id + h\partial )\inv (\id - h \delta)  -(\delta+\partial) h_\partial \Big) \\
%& =(\id - h \delta)(\id + \partial h)(\id + \partial h)\inv - (\id + h\partial) (\delta+\partial) h(\id + \partial h)\inv \\
%& =  \sigma (\id + \partial h)\inv.
%\end{align*} 
%For the last formula, it follows from the definitions that 
%$$
%\delta_\partial = \sigma_\partial \circ (\delta + \partial ) \circ \tau_\partial = \sigma_\partial \delta \tau_\partial + \sigma_\partial \partial \tau_\partial.
%$$
%One can show that  
%\begin{gather*}
%\sigma_\partial \delta \tau_\partial = \delta + \sigma \Big(\sum_{n =0}^\infty (-1)^{n+1} n \  \partial (h \partial)^{n} \Big)\tau,  \quad \text{and} \quad 
%\sigma_\partial \partial \tau_\partial =\sigma \Big(\sum_{n =0}^\infty (-1)^{n } (n+1) \  \partial (h \partial)^{n} \Big)\tau.
%\end{gather*}
%Thus, 
%\begin{align*}
%\delta_\partial & = \delta + \sigma \Big(\sum_{n =0}^\infty (-1)^{n } (n+1 -n) \  \partial (h \partial)^{n} \Big)\tau \\
%& = \delta +  \sum_{n =0}^\infty    \sigma \partial (-h \partial)^{n} \tau.
%\end{align*}
%This completes the proof.
%\end{proof}
%\end{comment}

\subsection{Hom spaces and tensor products of contractions}\label{sec:HomTensorConstruction}

In the study of Atiyah classes, it is necessary to construct a contraction with its big space being  
$$
W\dual \otimes_R \End_R(W) \cong W\dual \otimes_R W\dual \otimes_R W
$$
using a given contraction $(W, \delta; h)$ over a commutative dg ring $R$. 
This construction is commonly known as the ``tensor trick'' in the literature. 
For example, contractions of $(m,n)$-tensors can be found in \cite{MR3877426}. Also see \cite{MR2760671,MR4504932}. 
In this paper, we use the formulation of Hom spaces.

\begin{lemma}\label{lem:HomContraction}
Let $(W ,\delta; h)$ and $(W',\delta'; h')$ be contractions over a commutative dg ring $R$.
The triple 
$$
\big(\Hom_R(W , W' ), \; D ; \; H  \big)
$$
with
\begin{align*}
D(f) &  = \delta' \circ f - (-1)^{|f|}f \circ \delta, \\
H(f) & = h' \circ f + (-1)^{|f|} (\id_{W'} -[\delta',h']) \circ f \circ h
\end{align*}
is also a contraction over $R$.
Furthermore, if the contraction data associated with $(W ,\delta; h)$ and $(W',\delta'; h')$ are 
$$
\begin{tikzcd}
(V,d)
\arrow[r, " \tau", shift left,hook] &  (W ,\delta)
\arrow[l, " \sigma", shift left,two heads] \arrow[loop, "h",out=12,in= -12,looseness = 3]
\end{tikzcd},
\qquad\text{and}\qquad
\begin{tikzcd}
(V',d')
\arrow[r, " \tau'", shift left,hook] &  (W' ,\delta')
\arrow[l, " \sigma'", shift left,two heads] \arrow[loop, "h'",out=12,in= -12,looseness = 3]
\end{tikzcd}
$$
respectively, then the contraction data associated with $\big(\Hom_R(W , W' ), \; D ; \; H  \big)$ is
$$
\begin{tikzcd}
\big(\Hom_R(V , V' ),   \tilde D \big)
\arrow[r, " \cT", shift left,hook] &  \big(\Hom_R(W , W' ),   D \big) 
\arrow[l, " \Sigma", shift left,two heads] \arrow[loop, "H",out=8,in= -8,looseness = 3]
\end{tikzcd}
$$
with
\[ \cT(g)  = \tau' \circ g \circ \sigma, \qquad 
\Sigma(f)  = \sigma' \circ f \circ \tau, \qquad
\tilde D(g)   = d' \circ g - (-1)^{|g|}g \circ d, \]
for $g \in \Hom_R(V , V' )$ and $f \in \Hom_R(W , W' )$. 
\end{lemma}

Let $(W_i,\delta_i; h_i)$, $i=1,\cdots,n$, be contractions over $R$. 
It follows from the usual tensor trick that the triple
$$
\big(W_1 \otimes_R \cdots \otimes_R W_n, D^n; H^n \big)  
$$
with
\begin{align}
D^n & = \sum_{i=1}^n \id_{W_1} \otimes \cdots \otimes \id_{W_{i-1}}  \otimes\, \delta_i \otimes \id_{W_{i+1}} \otimes \cdots \otimes \id_{W_n}, \label{eq:diff_tensor} \\
 H^n   & = \sum_{i=1}^n (\id_{W_1} - [\delta_1,h_1]) \otimes \cdots \otimes (\id_{W_{i-1}} - [\delta_{i-1},h_{i-1}]) \otimes h_i \otimes \id_{W_{i+1}} \cdots \otimes \id_{W_n} \label{eq:HPT_tensor}
\end{align}
is also a contraction over $R$.
Furthermore, if the contraction data associated with $(W_i,\delta_i; h_i)$ is 
$$
\begin{tikzcd}
(V_i , d_i)
\arrow[r, " \tau_i", shift left,hook] &  (W_i ,\delta_i)
\arrow[l, " \sigma_i", shift left,two heads] \arrow[loop, "h_i",out=12,in= -12,looseness = 3]
\end{tikzcd},
$$
then the contraction data of the tensor product is 
$$
\begin{tikzcd}
\big(V_1 \otimes_R \cdots \otimes_R V_n, \tilde D^n \big)
\arrow[r, " \cT^n", shift left,hook] &  \big(W_1 \otimes_R \cdots \otimes_R W_n, D^n \big)
\arrow[l, " \Sigma^n", shift left,two heads] \arrow[loop, "H^n",out=5,in= -5,looseness = 3]
\end{tikzcd},
$$
where
\begin{gather*}
\cT^n  = \tau_1 \otimes \cdots \otimes \tau_n,  \qquad
 \Sigma^n  = \sigma_1 \otimes \cdots \otimes \sigma_n, \\
\tilde D^n  = \sum_{i=1}^n \id_{V_1} \otimes \cdots \otimes \id_{V_{i-1}}  \otimes\, d_i \otimes \id_{V_{i+1}} \otimes \cdots \otimes \id_{V_n}.
\end{gather*}

By Lemma~\ref{lem:HomContraction}, we have the following

\begin{proposition}\label{prop:HomTensorCont}
Let $(W_i,\delta_i; h_i)$, $i=1,\cdots,n$,
and $(\bar W_j,\bar\delta_j; \bar h_j)$, $j=1, \cdots, m$, be contractions over $R$.
Then the triple  
$$
\big(\Hom_R(W_1\otimes_R \cdots \otimes_R W_n, 
\bar W_1 \otimes_R \cdots \otimes_R \bar W_m), \; D_n^m; \; H_n^m \big)
$$
is also a contraction over $R$, where $D_n^m$ and $H_n^m$ are determined by 
\eqref{eq:diff_tensor}, \eqref{eq:HPT_tensor} and Lemma~\ref{lem:HomContraction}. 
Furthermore, if the  contraction data 
associated with $(W_i,\delta_i; h_i)$ and $(\bar W_j ,\bar\delta_j ; \bar h_j )$ are 
$$
\begin{tikzcd}
(V_i , d_i)
\arrow[r, " \tau_i", shift left,hook] &  (W_i ,\delta_i)
\arrow[l, " \sigma_i", shift left,two heads] \arrow[loop, "h_i",out=12,in= -12,looseness = 3]
\end{tikzcd}
\qquad\text{and}\qquad
\begin{tikzcd}
(\bar V_j  , \bar d_j)
\arrow[r, " \bar\tau_j ", shift left,hook] &  (\bar W_j  ,\bar \delta_j )
\arrow[l, " \bar\sigma_j ", shift left,two heads] \arrow[loop, "\bar h_j ",out=12,in= -12,looseness = 3]
\end{tikzcd}
$$
respectively, then the diagram 
$$
\begin{tikzcd}
\big( \tilde\TT_n^m  , \tilde D_n^m\big)
\arrow[r, " \cT_n^m", shift left,hook] &  \big(\TT_n^m ,D_n^m\big)
\arrow[l, " \Sigma_n^m", shift left,two heads] \arrow[loop, "H_n^m",out=10,in= -10,looseness = 3]
\end{tikzcd},
$$
with
\begin{gather*}
\tilde\TT_n^m  = \Hom_R(V_1\otimes_R \cdots \otimes_R V_n, \bar V_1 \otimes_R \cdots \otimes_R \bar V_m ), \\
\TT_n^m = \Hom_R(W_1\otimes_R \cdots \otimes_R W_n,\bar W_1 \otimes_R \cdots \otimes_R \bar W_m ), \\
\cT_n^m(g)  = \bar\cT^m \circ g \circ \Sigma^n, \qquad
\Sigma_n^m(f)  = \bar\Sigma^m \circ f \circ \cT^n. 
\end{gather*}
is a contraction data.
\end{proposition}

Let $(W,\delta;h)$ be a contraction over $R$ whose associated contraction data is 
$$
\begin{tikzcd}
(V , d)
\arrow[r, " \tau", shift left,hook] &  (W ,\delta)
\arrow[l, " \sigma", shift left,two heads] \arrow[loop, "h",out=12,in= -12,looseness = 3]
\end{tikzcd}.
$$
By \eqref{eq:diff_tensor} and \eqref{eq:HPT_tensor}, one has the contraction $(\TT W, D_{\TT}; H_{\TT})$ of the tensor algebra $\TT W$ generated by $W$ over $R$.  
Let $\Lambda W$ be the exterior algebra generated by $W$ over $R$, 
$D_\Lambda: \Lambda W \to \Lambda W$ be the derivation generated by $\delta$, and $H_\Lambda$ be the operator defined by the composition
$$
H_\Lambda: \Lambda W \xto{\ssym} \TT W \xto{H_\TT} \TT W \onto \Lambda W,
$$
where $\ssym: \Lambda W \to \TT W$ is the map 
$$
\ssym(w_1 \wedge \cdots \wedge w_n)= \dfrac{1}{n!} \sum_{\sigma \in S_n} \chi_\sigma \cdot w_{\sigma(1)} \otimes \cdots \otimes w_{\sigma(n)}
$$
--- $S_n$ denotes the symmetric group on the set $\{1,2,\cdots,n\}$,
and $\chi_\sigma$ is the number (either $+1$ or $-1$) satisfying the equation
$$
w_1 \wedge \cdots \wedge w_n = \chi_\sigma \cdot w_{\sigma(1)} \wedge \cdots \wedge w_{\sigma(n)}
$$
in $\Lambda W$. In other words, 
\begin{multline*}
H_\Lambda(w_1 \wedge \cdots \wedge w_n) = \dfrac{1}{n!} \sum_{\sigma \in S_n} \sum_{i=1}^n (-1)^{|w_{\sigma(1)}|+\cdots +|w_{\sigma(i-1)}|} \chi_\sigma \cdot (\id - [\delta,h])(w_{\sigma(1)}) \wedge \cdots \\
 \wedge (\id - [\delta,h])(w_{\sigma(i-1)}) \wedge h(w_{\sigma(i)}) \wedge w_{\sigma(i+1)} \cdots \wedge w_{\sigma(n)}.
\end{multline*}

\begin{proposition}\label{prop:ExteriorAlgContraction}
The triple $(\Lambda W, D_\Lambda; H_\Lambda)$ is a contraction over $R$  
with the contraction data 
$$
\begin{tikzcd}
(\Lambda V,   \tilde D_\Lambda )
\arrow[r, " \cT_\Lambda", shift left,hook] &  (\Lambda W, D_\Lambda ) 
\arrow[l, " \Sigma_\Lambda", shift left,two heads] \arrow[loop, "H_\Lambda",out=8,in= -8,looseness = 3]
\end{tikzcd},
$$ 
where $\tilde D_\Lambda$ is the derivation generated by $d$, and 
\begin{align*}
\cT_\Lambda(v_1 \wedge \cdots \wedge v_n) &= \tau(v_1) \wedge \cdots \wedge \tau(v_n), \\
\Sigma_\Lambda(w_1 \wedge \cdots \wedge w_n) & = \sigma(w_1) \wedge \cdots \wedge \sigma(w_n),
\end{align*}
for $v_1, \cdots, v_n \in V$, $w_1, \cdots, w_n \in W$. 
\end{proposition}

We need the following contraction when studying the Todd classes. 

\begin{corollary}\label{cor:ToddContraction}
The triple $(\Lambda W\dual \otimes \End W, \widehat D ;\widehat H)$ is a contraction over $R$, where $\widehat D$ and $\widehat H$ are given by Lemma~\ref{lem:HomContraction} and Proposition~\ref{prop:ExteriorAlgContraction}.  Furthermore, the induced inclusion map
$$
\widehat{\cT}: \Lambda V\dual \otimes \End V \into \Lambda W\dual \otimes \End W
$$
is an algebra morphism. 
\end{corollary}
\begin{proof}
The first assertion is clear. 
For the second one, since $\sigma \circ \tau =\id_V$, if follows from Lemma~\ref{lem:HomContraction} that $\cT(g)\circ \cT(g') = \cT(g \circ g')$. Thus, by Proposition~\ref{prop:ExteriorAlgContraction}, the tensor trick implies  $\widehat{\cT} =  \cT_\Lambda\dual \otimes \cT$ is an algebra morphism. 
\end{proof}
Note that since $\tau \circ \sigma \neq \id_W$, the projection map $\Sigma: \End W \to \End V$ is not necessarily an algebra morphism.

\bibliographystyle{plain}
\bibliography{refAtiyah}

\end{document}